\documentclass[10pt]{article}
\usepackage{fullpage}
\usepackage{amsmath}
\usepackage{amssymb}
\usepackage{amsthm}
\usepackage{array}
\usepackage{xy}
\usepackage{epsfig}
\usepackage{setspace}
\usepackage{graphics}
\usepackage{wrapfig}
\usepackage{mathrsfs}
\usepackage{ifthen}
\usepackage{verbatim}
\usepackage{color}
\usepackage{mathrsfs}
\usepackage{bm}
\usepackage{bbm}
\usepackage[authoryear]{natbib}
\usepackage{lmodern}
\usepackage{appendix}
\usepackage[dvipsnames]{xcolor}
\usepackage{tikz}
\usetikzlibrary{calc}
\usetikzlibrary{arrows}
\usetikzlibrary{matrix,chains,positioning,decorations.pathreplacing,arrows}
\usepackage[tight]{subfigure} 
\usepackage{caption}

\usepackage{hyperref}
\usepackage{xcolor}\hypersetup{allcolors=blue,colorlinks=true}
\usepackage{tcolorbox}
 \usepackage{stackengine}
 \usepackage{float} 
\renewcommand{\P}{\mathbf{P}}
\newcommand{\ind}{\mathbbm{1}}

\newcommand{\D}{\mathbbm{D}}
\newcommand{\E}{\mathbf{E}}

\newcommand{\R}{\mathbb{R}}
\newcommand{\N}{\mathbb{N}}

\newcommand{\lip}{\mathrm{Lip}}
\newcommand{\one}{\mathbbm{1}}

\usepackage[normalem]{ulem}

\newcommand{\vertiii}[1]{{\left\vert\kern-0.25ex\left\vert\kern-0.25ex\left\vert #1
    \right\vert\kern-0.25ex\right\vert\kern-0.25ex\right\vert}}
\usepackage{hyperref}

\numberwithin{equation}{section}

\newtheorem{theorem}{Theorem}
\newtheorem{lemma}{Lemma}
\newtheorem{proposition}{Proposition}

\newtheorem{definition}[lemma]{Definition}

\newtheorem{assumption}{Assumption}

\newtheorem{result}{Result}

\numberwithin{lemma}{section}
\numberwithin{theorem}{section}
\numberwithin{result}{section}
\numberwithin{proposition}{section}

\DeclareMathOperator*{\argmin}{arg\,min}

\usepackage{marvosym, wasysym}
\usepackage{fancyhdr}
\usepackage{stackengine}

\usepackage[margin=1.2in]{geometry}

\usepackage{lipsum}
\newif\iftheoremtree
\newif\ifnotationindex
\newif\ifrebuttal
\usepackage[shortlabels]{enumitem}
\newlist{thmdependence}{itemize}{10}
\setlist[thmdependence]{nosep,label=-}

\newcommand{\thmtreeref}[2]{\item[\elsewhere] {{\hyperlink{thm tree #2}{\color{gray}#1}}}~\ref{#2}}
\newcommand{\thmtreenode}[5]{\item[#1] \linkdest{thm tree #3} {#2}~\ref{#3} \thmsum{#4}{#5}}
\newcommand{\thmtreenodeproof}[5]{\item[#1] \linkdest{thm tree #3} {#2}~\ref{#3} \linktoproof{proof of #3} \thmsum{#4}{#5}}

\newcommand{\linktoproof}[1]{\hyperlink{#1}{\pflinksymbol}}
\renewcommand{\linktoproof}[1]{}
\newcommand{\pflinksymbol}{{\tiny [Proof]}}

\newcommand{\complete}{{\color{black}\checkmark}}

\newcommand{\elsewhere}{}

\newcommand{\CR}[1]{{\color{red}[CR: #1]}} % comments by Chang-Han Rhee
\renewcommand{\CR}[1]{} % comments by Chang-Han Rhee
\newcommand{\CRA}[1]{{\color{red}\sout{[CR: #1]}}} % comments by C.-H.R. that are already addressed
\renewcommand{\CRA}[1]{} % comments by C.-H.R. that are already addressed
\newcommand{\XWA}[1]{{\color{RoyalBlue}\sout{[XW: #1]}}} % comments by X.W. that are already addressed 
\renewcommand{\XWA}[1]{} % comments by X.W. that are already addressed 

\newcommand{\thmsum}[2]{\quad{\color{gray}\begin{minipage}[t]{#1\linewidth}{#2}\vspace{0.5\baselineskip}\end{minipage}}}

\makeatletter
\newcommand{\linkdest}[1]{\Hy@raisedlink{\hypertarget{#1}{}}}
\makeatother

\ifrebuttal
    \newcommand{\rvtxt}[3]{\linkdest{#1}\hyperlink{back to #2}{\color{red}#3}%
    }
\else
    \newcommand{\rvtxt}[3]{#3}%
\fi

\newcommand{\rvidx}[2]{\linkdest{link to #1}\hyperlink{#1}{#2}}
\newcommand{\A}[1]{{\bf\color{NavyBlue}{#1}}}

\usepackage{mathtools}

\title{Large deviations for stochastic fluid networks\\ with Weibullian tails}
\author{
       Mihail Bazhba\thanks{Department of Quantitative Economics, UvA, Netherlands, Roetersstraat 11, {m.bazhba@uva.nl}} \ 
       Chang-Han Rhee\thanks{Department of Industrial Engineering, Northwestern University, US, IL 60208-3109, {chang-han.rhee@northwestern.edu}} \ \ Bert Zwart\thanks{{Department of Stochastics, CWI, Netherlands, Science Park 123, {bert.zwart@cwi.nl and Eindhoven University of Technology}}}\\\\
}
\date{\today}

\begin{document}

\maketitle

\abstract{
We  consider a stochastic fluid network where the external input processes are compound Poisson with heavy-tailed Weibullian jumps. Our results comprise of large deviations estimates for the buffer content process in the vector-valued Skorokhod space which is endowed with the product $J_1$ topology. To illustrate our framework, we provide explicit results for a tandem queue. At the heart of our proof is a recent sample-path large deviations result, and a novel continuity result for the Skorokhod reflection map in the product $J_1$ topology.  



}

\section{Introduction}\label{intro-stochastic-fluid-network}

% 	The stochastic network is a key model within applied probability and is connected to many applications. Some real-life examples include  computer communication and manufacturing networks.

The past 25 years have witnessed a significant research activity on queueing systems with heavy tails, but the important case of queueing networks has received less attention.
% Two early papers focused on monotone separable networks \cite{BaFo2004} and max-plus networks \cite{BaLaFo04}.
Early papers focused on generalised Jackson networks (\cite{baccelli2004tails}), monotone separable networks (\cite{BaFo2004}), and max-plus networks (\cite{BaLaFo04}). Recent work on tail asymptotics of transient cycle times and waiting times for closed tandem queueing networks can be seen in \cite{kim2015cyclic}.
In two joint papers with Foss, Masakiyo Miyazawa investigated queue lengths in a queueing network with feedback in \cite{FoMi14} and tandem queueing networks in \cite{FoMi18}. Compared to standard queueing networks tracking movements of discrete customers, fluid networks are somewhat more tractable. In an early paper, it was recognized that tail asymptotics for downstream nodes could be obtained by analyzing the busy period of upstream nodes, under certain assumptions (\cite{BoxmaDumas}). The case of a tandem fluid queue where the input to the first node is a L\'evy process with regularly varying jump sizes has been investigated in \cite{Lieshout} exploiting a Laplace transform expression which is available in that case. 

More recently, multidimensional asymptotics for the time-dependent buffer content vector in a fluid queue fed by compound Poisson processes were investigated in \cite{Chen2019}. 
The framework in \cite{Chen2019} allows for the analysis of situations in which a large buffer content may be caused by multiple big jumps in the input process. 
Such results were established before for multiple server queues and fluid queues fed by on-off sources (see, for example, \cite{zwart2004exact}, 
\rvtxt{fossref}{bullet}
{\cite{foss2006heavy}}%
%\hyperlink{bullet}{\color{red}\tiny[back to the rebuttal]}
, 
\cite{foss2012large}). 
The results on fluid networks in \cite{Chen2019} were derived assuming regular variation of the jumps in the arrival processes.
%Asymptotic bounds for fluid queues with sub-exponential on-off sources have been explored in \cite{dumas2000asymptotic}. In the case of regularly varying on-periods, sharp asymptotics were derived in \cite{zwart2004exact}. 
Work on fluid networks with light-tailed input is surveyed in \cite{miyazawa2011light}. The goal of the present paper is to investigate the case where jumps are semi-exponential (e.g.\ of Weibull type $\exp\{-x^\alpha\}$ with $\alpha \in (0,1)$). This case is somewhat more difficult to analyze, especially in the case where rare events of interest are caused by multiple big jumps in the input process, as exhibited in the case of the multiple server queue (\cite{bazhba2019queue}).

We focus on a stochastic fluid network comprised of $d$ nodes, with external inputs modeled as compound Poisson processes  with semi-exponential increments.
We are interested in the event that an arbitrary linear combination of the buffer contents in the network exceeds a large value.  We write this functional as a mapping of the input processes using the well-known multidimensional Skorokhod reflection map on the positive orthant (see e.g.\ \cite{whitt2002stochastic}), and apply a sample-path large deviations principle for the superposition of Poisson processes, which has recently been derived in \cite{bazhba2020sample}. This  sample-path large deviation principle has been established for Poisson processes with semi-exponential jumps, and  holds in the product $J_1$ topology. To apply the contraction principle (the analogue of the continuous mapping argument in a large deviations context), we need to show that the Skorokhod map has suitable continuity properties. The $J_1$ product topology is not as strong as the standard $J_1$ topology on $\R^d$, and it turns out that continuity can only be established for input processes with nonnegative jumps. However, this result, presented in Theorem~\ref{lipschitz-continuity-phi} below, is sufficient for our proof strategy to work.

The contraction principle leads to an expression of the rate function which we analyze in detail. Under some generality, we show that the upper and lower bound of the large deviations bounds match. 
We conjecture that each input process contributes to a large fluid level by a finite number of big jumps, and the computation of the rate function can be reduced to a concave optimization problem with a finite number of decision variables. We illustrate this by reducing the optimization problem to a finite dimensional problem and then explicitly solving it for the case $d=2$ in Section~\ref{tandem}.

The outline of this paper is as follows: Section~\ref{model-description-preliminaries} contains a description of our model, the topological space in which the input processes are defined, and an introduction to the reflection map. In Sections~\ref{SFN-SP-BCP}, \ref{SFN-overflow-prob}, and \ref{tandem} we present our main results: upper and lower large deviation bounds for the buffer content process, logarithmic asymptotics for overflow probabilities of the buffer content process over fixed times, and an explicit analysis of the two-node tandem network.  %study of a variation of the multiple on-off sources model. In the latter caae, we include an explicit computation of the decay rate. 
Section~\ref{proofs-stochastic-networks} contains technical proofs. We end this paper with an appendix where we develop several auxiliary large deviations results. % that facilitate the use of the extended LDP. EXTENDED LDP NOT DEFINED YET
		
\section{Model description and preliminary results}
\label{model-description-preliminaries}

\subsection{The Model}
In this section, we describe our model and we present some preliminary results that are used in our analysis.
We consider a single-class open stochastic fluid network with $d$ nodes. 
% Consider a fixed time horizon $T>0$. 
We denote the total amount of external work that arrives at station $i$ with \linkdest{nota-J_i-i-in-mathcal-J}$J_i(t) \triangleq \sum_{j=1}^{N_i(t)}J^{(j)}_i$ which is a compound Poisson process with mean \linkdest{nota-gamma}$\gamma_i$ where $\{J^{(j)}_i\}_{j=1,2,\ldots}$ is an iid jump size sequence for each $i=1,\ldots, d$. 
%\M{and $T$ denotes the right-end boundary of the time horizon $[0,T]$ upon which the process $J^{(i)}$ is defined}. 
If no exogenous input is assigned to node $i$, then we set \linkdest{nota-J_i-i-notin-mathcal-J}$J_i(\cdot) \equiv 0$, and \linkdest{nota-gamma_i}$\gamma_i \triangleq 0$. 
We define \linkdest{nota-mathcal-J}$\mathcal{J}$ as the subset of nodes that have an exogenous input.
\rvtxt{(a)}{(a)}
{%
We assume that $\{J_1(t): t\geq 0\}, \{J_2(t): t\geq 0\}, \ldots, \{J_d(t): t\geq 0\}$'s are independent.%
}
For notational convenience, we assume that the Poisson processes \linkdest{nota-N_i}$\{N_i(t)\}_{t \geq 0}$ have unit rate for $i \in \mathcal{J}$.  
The key assumption on the distribution of the jump sizes $J^{(1)}_i$ for $i \in \mathcal{J}$ is that they are semi-exponential:
\begin{assumption}\label{right-tail-W-i}
For each $i \in \mathcal{J} \subseteq \{1,\ldots,d\}$,  $\P\big(J^{(1)}_i \geq x \big) = e^{-c_i L(x)x^{\alpha}}$ where \linkdest{nota-alpha}$\alpha \in (0,1)$,  \linkdest{nota-c_i}$c_i \in (0,\infty),$ and $L$ is a slowly varying function such that $L(x)/x^{1-\alpha}$ is non-increasing for sufficiently large $x$.  
\end{assumption}
Recall that $L$ is slowly varying if 
\rvtxt{1}{1}{%
\linkdest{nota-L}$L(ax)/L(x)\rightarrow 1$%
}
as $x\rightarrow\infty$ for each $a>0$.
At each node $i \in \{1,\ldots,d\}$, the fluid is processed and released at a deterministic rate $r_i$.
Fractions of the processed fluid from each node are then routed to other
nodes or leave the network. We characterize the stochastic fluid network  by a four-tuple \linkdest{nota-bm-J}$(\bm{J}, \bm{r}, Q, \bm{X}(0))$, where $\bm{J}(\cdot) = \big(J_1(\cdot),\ldots,J_d(\cdot)\big)$ is the vector of the assigned
 input  processes at each one of the $d$ nodes, respectively.  
%The random variable $J^{(i)}(t)$ represents the total amount of 
%exogenous input to node $i$ during the time interval $[0,t]$. 
The vector \linkdest{nota-bm-r}$\bm{r} \triangleq  (r_1 ,\ldots , r_d )^\intercal$ is the vector of
deterministic output rates at the $d$ nodes, \linkdest{nota-Q}$Q \triangleq [q_{ij}]_{i,j \in \{1,\ldots,d\}} $
% \CR{inconsistent notation with, for example, line 262 (of the tex source) where we write $(q_{ij})_{i,j}$} \M{OK}  
is a $d \times d$ substochastic routing matrix, and
\linkdest{nota-bm-X(0)}$\bm{X}(0) \triangleq (X_1 (0),\ldots, X_d(0))$ is a nonnegative random vector of initial contents at the $d$ nodes. 
%Now, we make our model more specific. Regarding the $d$-dimensional stochastic fluid model, we assume that the input flow streams to the $i$-th station/node are Poisson processes with unit rate.
%THIS IS NOT NEEDED AND QUITE ARBRITRARY BUT OK
%
%Let $\{N^{(i)}(t)\}_{t \geq 0}$ denote the Poisson process of unit rate that is associated with each station $i$, which is also independent from $\{N^{(j)}(t)\}_{t \geq 0}$ for every $j=1,\ldots,d$. At each node, the arrival of the $k$th job in station $i$ generates a workload $J^{(i)}_k$. In addition, let $\bm{J}_k=(J_k^{(1)},\ldots,J_k^{(i)},\ldots,J_k^{(d)})^\intercal$ denote a sequence of i.i.d. positive random vectors with i.i.d. increments such that $\{\bm{J}_k\}_{k \geq 1}$ is independent of $\{N^{(i)}(t)\}_{t \geq 0}$, for each $i \in \{1,\ldots,d\}$. 
%Naturally, the stochastic process $\bm{J}$ is non-decreasing, non-negative, and its sample paths are allowed to be discontinuous.
%
% Consequently, the stochastic process $\bm{J}$ is an element of $\prod_{i=1}^{d}\D^{\uparrow}[0,1]$---the subset of functions over the domain $[0,1]$ with co-domain the space of real numbers that are non-decreasing and non-negative in each coordinate.
If the buffer at node $i$ and at time $t$ is nonempty, then there is fluid output from node $i$ at a
constant rate $r_i$. On the other hand, if the buffer of node $i$ is empty at time $t$, the output rate equals the minimum
of the combined (i.e., both external and internal) input rates and the 
output rate $r_i$.

We now provide more details on the stochastic dynamics of our network. A proportion $q_{ij}$ of all output from node $i$ is
immediately routed to node $j$, while the remaining proportion $q_i \triangleq   1-\sum_{j=1}^{k}q_{ij}$ leaves 
the network. We assume that $q_{ii} \triangleq 0$, and the routing matrix $Q $ is substochastic, so that $q_{ij} \geq  0$,
and $q_i \geq 0$ for all $i, j$. We also assume that $Q^n \to 0$ as $n \to \infty$ which implies that all input eventually leaves the network.
Let $Q^\intercal$ be the transpose matrix of $Q$. 
Though we focus on time-dependent behavior, we consider the scenario that the fluid network is stable, ensuring that a high level of fluid is a rare event. 
Let \linkdest{nota-mathcal-Q}$\mathcal{Q}=(\mathrm{I}-Q^\intercal)$.
We guarantee the stability of the network by posing the following assumption, based on \cite{kella1996stability}:
 
\begin{assumption}\label{stability-assumption}
Let \linkdest{nota-bm-gamma}$\bm{\gamma}=(\gamma_1,\ldots,\gamma_d)^\intercal$, and assume that $\bm{r} > \mathcal Q^{-1}\boldsymbol{\gamma}$.
\end{assumption}

% We start by defining a potential buffer-content (or
% net-input) process, which represents the potential content at each node, ignoring
% the emptiness condition.

Due to our model specifics, the buffer content at station $i$ is processed at a constant rate $r_i$ from the $i$-th server; and a proportion $q_{ij}$ is routed from the $i$-th station to the $j$-th server. 
% Let $\mathcal{Q}=(\mathrm{I}-Q^\intercal)$.  
To define the buffer content process we first define the potential content
vector $\bm{X}(t)$ 
$$
\linkdest{nota-bm-X(t)}\bm{X}(t) \triangleq \bm{X}(0)+ \bm{J}(t)-\mathcal{Q}\bm{r}\cdot t, \quad t\geq 0.
$$
%, at time $t$, would be the initial value $\bm{X}(0)$ plus the exogenous input
%$\bm{J}(t)$ minus the output $r \cdot t$ plus the internal input $Qr\cdot t$.   
Let \linkdest{nota-bm-Z}$\bm{Z}_i(t)$ denote the buffer content of the $i$-th station at time $t$. We can define the buffer content process by the so-called reflection map. We first provide an intuitive description of this map. It is defined in terms of a pair of processes $(\bm{Z},\bm{Y})$ that solve the  differential equation
\begin{equation}\label{Skorohod-stoc-dif-eq}
    d\bm{Z}(t)=d\bm{X}(t)+\mathcal{Q}d\bm{Y}(t), \ t\geq 0.
\end{equation}
Here, \linkdest{nota-bm-Y}$\bm{Y}(\cdot)$ is non-decreasing  and ${\bm Y}_i(t)$ only increases at times where $\bm{Z}_i(t)=0$ for all $i$ and all $t$.  %The component $X^{(i)}(t)$ represents what the content of buffer $i$
%would be at time $t$ if the output occurred continuously at rate $r_j$ from node $j$, for
%all $j$, whether station $j$ had fluid to emit.
% We obtain the actual buffer content by disallowing the potential output and the 
% associated internal input that cannot occur because of emptiness.
Consequently, as we assume $\bm{Z}(0)=0$, the buffer content is
\begin{equation}\label{buffer-content}
\bm{Z}(t) = \bm{X}(t)+\mathcal{Q} \bm{Y}(t), \ t \geq 0.
\end{equation}
 We call the map  $\bm{X} \mapsto (\bm{Y}, \bm{Z})$ the reflection map. We now provide a more rigorous definition of this map.

\subsection{The reflection map with discontinuities}\label{SFN-prelim-reflection-map}
We start with the definition of the reflection map.
Fix an arbitrary $T>0$. 
Let \linkdest{nota-D[0,T]}$\D[0,T]$ denote the Skorokhod space: the space of c\`adl\`ag paths over the time horizon $[0,T]$. 
Note that for our large deviations analyses, we will consider linearly scaled processes in $\D[0,T]$, and hence, this translates considering the time horizon $[0,nT]$ for the original unscaled processes. 
Denote with \linkdest{nota-D^uparrow}$\D^{\uparrow}[0,T]$ the subspace of the Skorokhod space consisting of non-decreasing functions that are non-negative at the origin. 
Note that we use the component-wise partial order on $\D[0,T]$ and $\R^d$. That is, we write
$\bm x \triangleq (x_1,\ldots, x_d) \leq \bm y \triangleq (y_1,\ldots,y_d)$ in $\R^d$ if $x_i \leq y_i$ in $\R$ for all $i \in \{1,\ldots,d\}$, and we write $\bm \xi \triangleq (\xi_1,\ldots,\xi_d) \leq \bm \zeta \triangleq (\zeta_1,\ldots,\zeta_d)$ in $\prod_{i=1}^{d}\D[0,T]$ if $\bm\xi(t)\leq\bm \zeta(t)$ in $\R^d$ for all $t\in [0,T]$.

%\begin{itemize}
%\item Denote the ; let 
%\item $\psi(\xi) \triangleq \inf \{\Psi(\xi)\}$ be the regulator component; 
%\item $\phi:\prod_{i=1}^{k}\D[0,T] \to \prod_{i=1}^{k}\D[0,T]$ is such that $\phi(\xi) \triangleq \xi +(I-Q)^\intercal\psi(\xi)$; and
%\item The map $R=(\phi,\psi)$ is the reflection map. 
%\end{itemize}

\begin{definition}\label{definitionofreflectionmap} (Definition 14.2.1 of \cite{whitt2002stochastic})
For any $\bm\xi \in \prod_{i=1}^{d}\D[0,T]$ and any reflection matrix $\mathcal{Q}=(\mathrm{I}-Q^\intercal)$, let the feasible regulator set be 
$$
\linkdest{nota-Psi}\Psi(\bm\xi) \triangleq \left\{\bm\zeta \in \prod_{i=1}^{d}\D^{\uparrow}[0,T]: \bm\xi+\mathcal{Q}  \bm\zeta \geq 0\right\},
$$
and let the reflection map be 
% \CR{Throughout the paper, we use `\textbackslash mapsto' in places of `\textbackslash to' when specifying the domain and codomain of mappings. For what I understand, `\textbackslash mapsto' is saved for defining function values; that is, for example, a function $f(x) = x^2$ from $\R$ to $\R$ is denoted as 
% \begin{align*}
% f:\R \to \R\\[-6pt]
% x \mapsto x^2.
% \end{align*}
% Another common usage is to denote the function itself. That is, we can write $x\mapsto x^2$ to refer to the function $f(x)= x^2$ without giving it a specific name $f$. 
% But it seems to be very rare to use `\textbackslash mapsto' in place of `\textbackslash to'.
% }
$$\linkdest{nota-bm-R}\bm{R} \triangleq (\psi,\phi): \prod_{i=1}^{d}\D[
0,T] \to \prod_{i=1}^{d}\D[
0,T]\times \prod_{i=1}^{d}\D[
0,T], $$ with regulator component 

$$\linkdest{nota-psi}\psi(\bm\xi) \triangleq \inf \left\{\Psi(\bm\xi)\right\} = \inf\left\{\bm w \in \prod_{i=1}^{d}\D[0,T]: \bm w \in \Psi(\bm\xi)\right\},$$
and content component 
$$\phi(\bm\xi) \triangleq \bm\xi+ \mathcal{Q}  \psi(\bm\xi).$$
\end{definition}
The infimum in the definition of  $\psi$ may not exist in general. However, in Theorem 14.2.1 of \cite{whitt2002stochastic}, it is proven that the reflection map is properly defined with the component-wise order. That is,
$$\psi_i(\bm\xi)(t) = \inf\{\omega_i(t) \in \R: \bm \omega \in \Psi(\xi)\} \ \text{for all} \ i\in \{1,\ldots,d\} \ \text{and} \ t \in [0,T].$$
In addition, the regulator set $\Psi(\bm \xi)$ is non-empty and its infimum is attained in $\Psi(\bm\xi)$ itself.
%If $\bm{R}=(\psi,\phi)$ is a continuous map, then the reflection map solves the Skorokhod problem implied by (\ref{Skorohod-stoc-dif-eq}). 
%this is a very weird sentence. you already established the mapping is well defined and solves what it is supposed to solve. continuity is something that comes next. better to delete the sentence IMO
Now, we state some important results regarding the properties of $(\phi,\psi)$. The following result gives an explicit representation of the solution of the Skorokhod problem. 

\begin{result} (Theorem 14.2.1, Theorem 14.2.5 and Theorem 14.2.7  of \cite{whitt2002stochastic})
\label{continuity-of-the-reflection-map}
  If $\bm Y(\cdot)=\psi(\bm{X})(\cdot)$ and $\bm{Z}(\cdot)=\phi(\bm{X})(\cdot)$, then $(\bm{Y}(\cdot),\bm{Z}(\cdot))$ solves the Skorokhod problem associated with the equation (\ref{Skorohod-stoc-dif-eq}). The mappings $\psi$ and $\phi$ are Lipschitz continuous maps w.r.t.\ the uniform metric.
\end{result}

The next result is a useful property of the Skorokhod map. It allows us to describe the discontinuities of the reflection map under some mild assumptions. 

\begin{result}\label{size-discontinuities-of-phi} (Lemma 14.3.3, Corollary 14.3.4 and Corollary 14.3.5 of \cite{whitt2002stochastic})
	Consider $\bm\xi \in \prod_{i=1}^{d}\D[0,T]$. 
	Let $Disc(\psi(\bm\xi))$ and $Disc(\phi(\bm\xi))$ denote the sets of discontinuity points of $\psi(\bm\xi)$ and $\phi(\bm\xi)$, respectively. 
	Then it holds that $Disc(\psi(\bm\xi)) \cup Disc(\phi(\bm\xi))=Disc(\bm\xi)$. 
	In addition, if $\bm\xi$ has only positive jumps, then $\psi(\bm\xi)$ is continuous and 
	\[
	\phi(\bm\xi)(t)-\phi(\bm\xi)(t-)=\bm\xi(t)-\bm\xi(t-).
	\]
\end{result}

\begin{result}\label{generic-upperbound-for-psi} (Theorem 14.2.6 of \cite{whitt2002stochastic})
If $\bm\xi \leq \bm\zeta$ in $\prod_{i=1}^{d}\D[0,T]$, $T>0$, then $\psi(\bm\xi) \geq \psi(\bm\zeta)$.
\end{result}

\subsection{Topologies and large deviations}\label{SFN-prelim-ldps}

In this section, we introduce our preliminary results on sample-path large deviations for the input and the content process. 
We begin with setting the notation. 
\CRA{Introduce $\|\cdot\|_1$ and $\|\cdot\|$ first, and then $d_{J_1}$.}%
For any \linkdest{nota-bm-beta}$\bm \beta = (\beta_1,\ldots,\beta_d) \in \R^d$, let $\|\bm \beta\|_1$ denote the usual $\ell_1$-norm: \linkdest{nota-vector-norm}$\|\bm \beta\|_1 = \sum_{i=1}^d |\beta_i|$.
For \linkdest{nota-path-norm}$\bm\xi = (\xi_1,\ldots, \xi_d)\in\prod_{i=1}^d \D[0,T]$, let $\|\bm \xi\| \triangleq \sup_{t\in[0,T]} \|\bm\xi(t)\|_1$. 
For large deviations results, we mainly work with the $J_1$ topology on $\D[0,T]$, and it's product topology on $\prod_{i=1}^d \D[0,T]$. 
Recall that in $\D[0,T]$, $J_1$ topology $\mathcal T_{J_1}$ is the one induced by the $J_1$ metric $d_{J_1}$:
$$
d_{J_1}(\xi, \zeta) = \inf_{\lambda \in \Lambda[0,T]} \left(\sup_{t\in[0,T]}\big|\xi\circ \lambda(t) - \zeta(t)\big|\right) \vee \left(\sup_{t\in[0,T]}\big|\lambda(t) - e(t)\big|\right) = \inf_{\lambda \in \Lambda[0,T]} \|\xi\circ \lambda - \zeta\| \vee \|\lambda - e\|,
$$
for $\xi, \zeta \in \D[0,T]$, where \linkdest{nota-e}$e:[0,T]\to[0,T]$ is the identity map $t\mapsto t$, and \linkdest{nota-Lambda}$\Lambda[0,T]$ is the set of all increasing homeomorphisms from $[0,T]$ to $[0,T]$.
In order to study networks, we need to set a topology in the vector-valued function space. 
That is, we work in the functional space $(\prod_{i=1}^{d}\D[0,T],\prod_{i=1}^{d}\mathcal{T}_{J_1})$ which is a product space equipped with the product $J_1$ topology $\prod_{i=1}^{d}\mathcal{T}_{J_1}$, which is induced by the product metric $d_p$: %which in turn is defined in terms of the $J_1$ metric on $\D[0,T]$. 
%More precisely, for
%we have that
$$
\linkdest{nota-d-p}
d_p(\bm \xi,\bm\zeta) = \sum_{i=1}^{d}d_{J_1}(\xi_i,\zeta_i)
$$
for $\bm \xi, \bm \zeta \in \prod_{i=1}^{d}\D[0,T]$ such that $\bm \xi=(\xi_1,\ldots,\xi_d)$ and $\bm \zeta=(\zeta_1,\ldots,\zeta_d)$. 
Unless specified otherwise, all the topological properties discussed in this paper are w.r.t.\ the topology generated by $d_p$.
% Whenever we make a statement on convergence w.r.t. $(\prod_{i=1}^{k}\D[0,T],\prod_{i=1}^{k}\mathcal{T}_{J_1})$, we mean convergence w.r.t.\ $d_p$. 
% Closed and open sets are understood to be generated by $d_p$. 
% \CR{These definitions of $\leq$ should be moved to somewhere earlier. They are already used in the definition of $\psi$ and Result 2.3.}\M{Start of Section 2.2}
% The following definition formalises the conditions which state the convergence of a sequence of functions  with respect to the product  $J_1$ topology.
%\begin{definition}
%	Let $\xi_n  \in \left(\prod_{i=1}^{k}\D[0,T],\prod_{i=1}^{k}\mathcal{T}_{J_1}\right)$. Then, $d_{p}(\xi_n,\xi) \to 0$ if and only if $\xi^{(i)}_n \to \xi^{(i)}$ w.r.t. the $J_1$ metric---for every $i=1,\ldots,k$.
%\end{definition}

\subsubsection{Some useful continuous functions}

The following two lemmas are elementary. Their proofs are provided in Appendix \ref{appendix-continuity}. %\CR{This sounds like we are simply omitting the proofs. TODO: Mention where we provide the proofs.}
\begin{lemma}\label{Continuity-Upsilon-k-multi}
\linktoproof{proof of Continuity-Upsilon-k-multi}
For $\bm\beta \in \R^d$, let \linkdest{nota-Upsilon_bm-beta}$\Upsilon^{\bm\beta}: \prod_{i=1}^{d}\D[0,T] \to \prod_{i=1}^{d}\D[0,T]$ be such that $\Upsilon^{\bm\beta}(\bm\xi)(t)=\bm\xi(t)+\bm\beta\cdot t$. Then, 
\begin{itemize}
\item[i)] $\Upsilon^{\bm\beta}$ is Lipschitz continuous w.r.t.\ $d_p$, 
\item[ii)] $\Upsilon^{\bm\beta}$ is a homeomorphism.
\end{itemize}
\end{lemma}

\begin{lemma}\label{continuity-of-the-projection-map}
\linktoproof{proof of continuity-of-the-projection-map}
For any $\bm b \in \R^d$, the mapping $\bm \xi \mapsto \bm b^\intercal \bm\xi(T)$ from $ \prod_{i=1}^{d}\D[0,T]$ to $\R$ is Lipschitz continuous  w.r.t.\ $d_p$. 
\end{lemma}

A key step in our approach is to establish the Lipschitz continuity of the regulator and the buffer content maps w.r.t.\ $d_p$. 
This is executed in Proposition~\ref{j1-product-uniform-continuity-regulator} and Theorem~\ref{lipschitz-continuity-phi} below. 
Their proofs are provided in Section~\ref{proofs-stochastic-networks}. 
Recall that $\D^{\uparrow}[0,T]$ is the subspace of the Skorokhod space containing non-decreasing paths which are non-negative at the origin. 
We say that $\xi\in \D[0,T]$ is a pure jump function if 
\rvtxt{(2)-1}{(2)}{$\xi = \sum_{j=1}^\infty x^{(j)}\mathbbm{1}_{[u^{(j)},T]}$ for some $x^{(j)}$'s and $u^{(j)}$'s such that $x^{(j)}\in \R$ and $u^{(j)}\in[0,T]$ for each $j$, and the $u^{(j)}$'s are all distinct.} 
Let 
\linkdest{nota-D-leq-infty-uparrow}$\D^{\uparrow}_{\leqslant \infty}[0,T]$ be the subspace of $\D[0,T]$ consisting of non-decreasing pure jump functions that assume non-negative values at the origin.
Subsequently, let 
\linkdest{nota-D-leq-k-uparrow}$\D^\uparrow_{\leqslant k}[0,T] \triangleq \{\xi \in \D[0,T]: 
\rvtxt{(2)-2}{(2)}{\xi= \sum_{j=1}^{k}x^{(j)}\one_{[u^{(j)},T]}, \ x^{(j)} \geq 0, \ u^{(j)}} \in [0,T], \ j=1,\ldots,k \}$ be the subset of $\D^{\uparrow}_{\leqslant \infty}[0,T]$ containing pure jump functions of at most $k$ jumps.
In addition, for $\beta \in \R$, let 
\linkdest{nota-D_leqslant-k^beta}$\D^{\beta}_{\leqslant k}[0,T] \triangleq \{\zeta \in \D[0,T]: \zeta(t)= \xi(t)+\beta \cdot t, \ \xi \in \D^{\uparrow}_{\leqslant k}[0,T] \}$
and
\linkdest{nota-D_leqslant-infty^beta}$\D^{\beta}_{\leqslant \infty}[0,T] \triangleq \{\zeta \in \D[0,T]: \zeta(t)= \xi(t)+\beta \cdot t, \ \xi \in \D^{\uparrow}_{\leqslant \infty}[0,T] \}.$
Let \linkdest{nota-D_leqslant-k}$\D_{\leqslant k}[0,T]$ denote the subspace of $\D[0,T]$ consisting of paths with at most $k$ jumps, i.e.\ $\D_{\leqslant k}[0,T] = \{\xi\in \D[0,T]: |Disc(\xi)| \leq k\}$.
Finally, let \linkdest{nota-D^beta}$\D^{\beta}[0,T] \triangleq \{\zeta \in \D[0,T]: \zeta(t)= \xi(t)+\beta \cdot t, \ \xi \in \D^{\uparrow}[0,T] \}$.
% \CR{It turns out that $\D^\beta[0,T]$ meant two different things throughout the paper. Now I separated them into $\D^\beta[0,T]$ and $\D^\beta_{\leqslant \infty}[0,T]$. We need distinguish one another throughout the paper. When I was checking Lemma 6.1 and 6.4, they were correct with this definition. But, for example, in Lemma 4.1, the definitions were mixed up, although it seems to be salvageable.}
%$\D^{\beta}[0,T] \triangleq \{\zeta \in \D[0,T]: \zeta(t)= \xi(t)+\beta \cdot t, \ \xi \in \D^{\uparrow}[0,T] \}.$
%\CR{What we intend to definef here must be 
%$\D^{\beta}[0,T] \triangleq \{\zeta \in \D[0,T]: \zeta(t)= \xi(t)+\beta \cdot t, \ \xi \in \D^{\uparrow}_{\leqslant \infty}[0,T] \}.$}
% \BZ{it takes more time to write such an intention than to actually do it. anyhow, it is done. }

\begin{proposition}\label{j1-product-uniform-continuity-regulator}
Let $\bm \beta = (\beta_1,\ldots,\beta_d)\in \R^d$. The regulator map $\psi$ is Lipschitz continuous w.r.t.\ $d_p$ on $\prod_{i=1}^{d}\D^{\beta_i}[0,T]$ with Lipschitz constant at most $d(2d^2(2d+1)K\|\bm\beta\|_1 + Kd\vee 1)$. 
\end{proposition}
Since $\phi(\bm\xi) = \bm\xi + \mathcal Q \psi(\bm\xi)$, the following is a corollary of Proposition~\ref{j1-product-uniform-continuity-regulator}.
\begin{theorem}\label{lipschitz-continuity-phi}
Let $\bm \beta = (\beta_1,\ldots,\beta_d)\in \R^d$. The reflection map $\bm{R}=(\phi,\psi)$ is Lipschitz continuous w.r.t.\ $d_p$ on $\prod_{i=1}^{d}\D^{\beta_i}[0,T]$.
\end{theorem}

Note that the restriction of the domain to the paths without downward jumps is essential for this type of results to hold.
Since the order in which the jumps take place  matters for the action of the reflection map, we cannot ensure the continuity of the reflection map without such extra regularity conditions. The main difficulty arises with  paths which have jumps with different signs in multiple coordinates appearing almost simultaneously (K. Ramanan, personal communication). %\CR{Are we going to explain this?}).   %\BZ{probably no time}
 
\subsubsection{The extended sample-path LDP for the potential buffer content process}

%	Not in all cases a random process satisfies an LDP. To treat cases where the standard large deviation principle is difficult or impossible to obtain, we present the concept of the extended LDP.
We first review the notion of extended LDP. 
Let $(\mathbb S, d)$ be a metric space, and $\mathcal T$ denote the topology induced by the metric $d$.
Let $\{X_n\}$ be a sequence of $\mathbb S$-valued random variables.
Let $I$ be a non-negative lower semi-continuous function on $\mathbb S$, and $\{a_n\}$ be a sequence of positive real numbers that tends to infinity as $n\to \infty$.
\begin{definition}\label{definition-extended-LDP}
	The probability measures of $(X_n)$ satisfy an \emph{extended} LDP in $(\mathbb S, d)$ with speed $a_n$ and rate function $I$ if
	$$
	-\inf_{x\in A^\circ}I(x)
	\leq \liminf_{n\to\infty} \frac{\log \P(X_n\in A)}{a_n}
	\leq
	\limsup_{n\to\infty} \frac{\log \P(X_n\in A)}{a_n}
	\leq
	-\lim_{\epsilon\to 0}\inf_{x\in A^\epsilon}I(x)
	$$
	for any measurable set $A$.
\end{definition}
Here we denote \linkdest{nota-A^epsilon}$A^{\epsilon} \triangleq \{\xi\in \mathbb S: d(\xi, A) \leq \epsilon \}$ where $d(\xi, A) = \inf_{\zeta \in A} d(\xi, \zeta)$.
The notion of an extended LDP has been introduced in \cite{borovkov2010large} and is  useful in the setting of semi-exponential random variables, in which a full LDP is provably impossible, as shown in \cite{bazhba2020sample}.
One important implication of extended LDP is an analog of the contraction principle. 
In the context of the extended LDP, the contraction principle requires Lipschitz continuity as opposed to mere continuity; see Lemma~\ref{alternative-yes-and-not-extended-contraction-principle}.

The main results of this paper in Sections~\ref{SFN-SP-BCP}, \ref{SFN-overflow-prob}, and \ref{tandem} are based on such contraction principles coupled with an extended LDP associated with the probability measures of the input process $\bm J(\cdot)$. 
% Let us give an overview our approach on an intuitive level. 
Specifically, the time evolution of $\bm Z(\cdot)$ may be written as
\[
\bm{Z}(t)=\bm{J}(t)-\boldsymbol{\gamma}t+(
\boldsymbol{\gamma}
-\mathcal{Q}\bm{r})t+\mathcal{Q}\bm{Y}(t), \quad t \geq 0.
\]
% assuming $\bm J(0) = 0$.
Equivalently, if we consider the scaled and centered input process \linkdest{nota-bar-J_n}${\bar{\bm{J}}}_n(\cdot) \triangleq \frac1n\bm{J}(n\cdot)-\bm\gamma\cdot e(\cdot)$, scaled potential buffer content process \linkdest{nota-bm-X_n}$\bm{X}_n(\cdot)\triangleq\frac{1}{n}\bm{X}(n\cdot)$, scaled regulator $\bm{Y}_n \triangleq\frac1n\bm Y(n\cdot)$, and scaled buffer content $\bm{Z}_n \triangleq \frac1n \bm Z(n\cdot)$, then
$$
\bm{Z}_n(t)= \bar{\bm{J}}_n(t)+
\rvtxt{(3)-1}{(3)}{\bm \kappa }t 
+ \mathcal{Q}\bm{Y}_n(t),\quad t \geq 0,
$$ 
where 
$
\rvtxt{(3)-2}{(3)}{\bm{\kappa}}
\triangleq \bm{\gamma}-\mathcal{Q}\bm{r}$.
% Now, let $(\phi,\psi)$ be the reflection map for reflecting matrix $\mathcal{Q}$, and 
Note that
$
\bm{Z}_n 
= 
\phi(\bm X_n) 
=
\phi \circ \Upsilon^{
\rvtxt{(3)-3}{(3)}{\bm \kappa }
}(\bar{\bm{J}}_n) 
$.
% $\bar {\bm X}_n \triangleq \Upsilon^{\boldsymbol{\beta}}(\bar{\bm{J}}_n)(t) = \bar{\bm{J}}_n+ \boldsymbol{\beta}\cdot t$, then 
Therefore, an extended LDP for $\bm{Z}_n$ can be deduced from that of $\bm X_n$, which, in turn, can be deduced from that of $\bar {\bm J}_n$, 
if $\phi$ and $\Upsilon^{\rvtxt{(3)-4}{(3)}{\bm \kappa}}$ are Lipschitz continuous in $J_1$ topology. 
Hence, the Lipschitz continuity of the shifting operator $\Upsilon^{\rvtxt{(3)-5}{(3)}{\bm \kappa}}$ and the content component map $\phi$ proved earlier in this section will play pivotal roles in our approach.
% \section{Main results}\label{main-results-stochastic-networks}

Now we conclude this section with establishing the desired extended LDP for the multidimensional input process $\bar{\bm J}_n$ and the potential buffer content process $\bm X_n$ of the stochastic fluid network.
% Recall that $\bm{J}$ denotes the input process which is a vector of independent compound Poisson processes with mean vector $\boldsymbol{\gamma}$.
For any $\xi \in \D[0,T]$, let $$I(\xi)=\sum_{ \{t: \xi(t) \neq  \xi(t-)\}}\left( \xi(t)- \xi(t-)\right)^\alpha.$$  
The next result is an immediate consequence of 
Theorem~2.3 and Remark~2.2 in \cite{bazhba2020sample}, combined with Lemma~\ref{E-LDP-on-subspaces-full-measure}.

\begin{result}\label{sample-path-ldp-for-Jn} 
The probability measures of $\bar{\bm{J}}_n$ satisfy the extended LDP in  $\big(\prod_{i=1}^{d}\D^{-\gamma_i}[0,T], \prod_{i=1}^{d}\mathcal{T}_{J_1}\big)$ with speed $L(n)n^{\alpha}$ and rate function ${I}^{(d)}:\prod_{i=1}^{d}\D^{-\gamma_i}[0,T] \to [0,\infty]$, where
\begin{equation}\label{eq:rate-function-Id-initial}
\linkdest{nota-I^(d)}{I}^{(d)}(\bm\xi)=
\begin{cases}
\sum_{j \in \mathcal{J}}c_jI(\xi_j)  & \text{if}\quad \xi_j \in \D^{\uparrow}_{\leqslant \infty}[0,T]  \quad\text{for}\quad j \in \mathcal{J} 
%  \\ & 
\quad\text{and}\quad \xi_j \equiv 0 \quad\text{for}\quad j \notin \mathcal{J},
 \\
\infty & otherwise.\\
\end{cases}
\end{equation}
\end{result}
Next, recall that $\bm X_n = \Upsilon^{\rvtxt{(3)-6}{(3)}{\bm \kappa}}(\bar {\bm J}_n)$.
% with $\bm\kappa = \bm\gamma - \mathcal Q \bm r$.
% the potential buffer content vector 
% %\CR{The name `unreflected content vector' seems to be used here for the first time. Previously, we called $\bm X$ as the potential content vector.}% 
% is a function of the exogenous input plus the internal input; that is, $\bm{X}(t)=\bm{J}(t) -\mathcal{Q}\bm{r}\cdot t$. 
% We define the scaled version of the potential content vector  . 
% As we will explain in more detail in Section~\ref{SFN-SP-BCP}, the large deviations of $\bar X_n$ will play a key role in our development.
% Obviously, $\bm{X}_n$ is the image of $\bm{\bar J}_n$ where the map $\Upsilon^{\bm \beta}$ is applied.   
Due to Lemma~\ref{Continuity-Upsilon-k-multi}, $\Upsilon^{\rvtxt{(3)-7}{(3)}{\bm \kappa}}$  is Lipschitz continuous and is a homeomorphism with respect to the product $J_1$ metric. 
The following extended large deviation principle for $\bm{X}_n(\cdot)$ is a direct consequence of Result~\ref{sample-path-ldp-for-Jn} and \textit{ii)} of Lemma~\ref{alternative-yes-and-not-extended-contraction-principle}.

% 
%Let $\D^{(\boldsymbol{\gamma}-\mathcal{Q}r)_i}[0,T]$ the subspace of step functions with positive steps and slope $(\gamma-\mathcal{Q}r)_i$ 

\begin{result}\label{theorem:multi-d+1-content}
The probability measures of \,$\bm{X}_n$ satisfy an extended LDP in $\big(\prod_{i=1}^{d}\D^{-(\mathcal{Q}\bm{r})_i}[0,T], \prod_{i=1}^{d} \mathcal{T}_{J_1}\big)$ with speed $L(n)n^\alpha$ and with rate function\linkdest{above 2}
\begin{equation}
\linkdest{nota-tilde-I^(d)}
\tilde{I}^{(d)}(\bm\xi)=
\begin{cases}
\sum_{j \in \mathcal{J}}c_jI(\xi_j)
& \text{if}\quad \xi_j \in 
\rvtxt{2}{2}{
\D^{(\boldsymbol{\gamma}-\mathcal{Q}\bm{r})_j}_{\leqslant \infty}[0,T]
} %\\ & 
\quad\text{for}\quad j \in \mathcal{J} 
\\ 
& \quad\text{and}\quad \xi_j = -(\mathcal{Q}\bm{r})_{j}\cdot e %\\ & 
\quad\text{for}\quad j \notin \mathcal{J},  
\\
\infty & otherwise.\\
\end{cases}
\end{equation}
\end{result}
We are now ready to state our first main result in the next section.
% along with some additional useful lemmas. 

\section{Large deviations for the buffer content process}\label{SFN-SP-BCP}

In this section, we state large deviation bounds for the scaled buffer content process $\bm{Z}_n$.
%(\cdot) \triangleq \frac{1}{n} \bm{Z}(n \cdot)$ with $\bm{Z}$ defined in (\ref{buffer-content}).
% The reflection map enables us to represent the buffer content process  in terms of  the potential content process $\bm{X}_n(\cdot)$ and the map $\psi$ and $\phi$ by $\bm{Y}_n =\psi(\bm{X}_n)$ and $\bm{Z}_n =\phi(\bm{X}_n)$.
We apply an analogue of the contraction principle for extended LDP's (Lemma~\ref{alternative-yes-and-not-extended-contraction-principle}) to obtain  asymptotic estimates for the probability measures of $(\bm{Z}_n)$:
 
 %The large deviation bounds for the buffer content process are a consequence of Lemma~\ref{alternative-yes-and-not-extended-contraction-principle}, and the continuity of the reflection map  with respect to the product  $J_1$ topology (Theorem~\ref{lipschitz-continuity-phi}). 

\begin{theorem}\label{samplepathldpofstochasticnetworks}
The probability measures of $\bm{Z}_n$ satisfy:
%an
%the following upper and lower bounds in 
%$\left(\prod_{i=1}^{d}\D[0,T],\prod_{i=1}^{d}\mathcal{T}_{J_1}\right)$ with speed $L(n)n^\alpha$ and rate function $I_S$ given by 
\begin{itemize}
\item[i)] For any set $F$ that is closed  in $\big(\prod_{i=1}^{d}\D[0,T],\prod_{i=1}^{d}\mathcal{T}_{J_1}\big)$,
\[
\limsup_{n \to \infty}\frac{1}{L(n)n^{\alpha}}\log \P\left( \bm{Z}_n \in F\right) \leq -\lim_{\epsilon \to 0}\inf_{\bm\xi \in F^{\epsilon}}I_{\bm Z}(\bm\xi).
\]
\item[ii)] For set $G$ that is open in $\big(\prod_{i=1}^{d}\D[0,T],\prod_{i=1}^{d}\mathcal{T}_{J_1}\big)$,
\[
\liminf_{n \to \infty}\frac{1}{L(n)n^{\alpha}}\log \P\left( \bm{Z}_n \in G\right) \geq -\inf_{\bm\xi \in G}I_{\bm Z}(\bm\xi), 
\]
\end{itemize}
where 
\[
\linkdest{nota-I-S}
I_{\bm Z}(\zeta) =\inf\Big\{\tilde{I}^{(d)}(\bm\xi): \bm\zeta=\phi(\bm\xi), \ \bm\xi \in \prod_{i=1}^{d}\D^{-(\mathcal{Q}\bm{r})_i}[0,T] \Big\}
\\
=\inf\left\{\tilde{I}^{(d)}(\bm\xi): \bm\xi\in \phi^{-1}(\bm\zeta) \right\}.
\] 
\end{theorem}
Note that $I_{\bm Z}$ may not be lower semi-continuous, because $\tilde I^{(d)}$ is not a good rate function; see \cite{bazhba2020sample} for details.

\begin{proof}
%The large deviation bounds for the buffer content process are a consequence of Lemma~\ref{alternative-yes-and-not-extended-contraction-principle}, and the continuity of the reflection map with respect to the product $J_1$ topology. 
Theorem~\ref{lipschitz-continuity-phi} ensures that $\phi$ is Lipschitz continuous w.r.t.\ $d_p$. 
Therefore, the upper and lower bounds in i) and ii) follow immediately from the extended LDP for $\bm X_n$ (Result~\ref{theorem:multi-d+1-content}) 
and the (Lipschitz) contraction principle   (Lemma~\ref{alternative-yes-and-not-extended-contraction-principle}).
% and 
% Result~\ref{theorem:multi-d+1-content}.
%\CR{TODO: space}
\end{proof}
The function $I_{\bm Z}$ is the solution of a constrained minimization problem over step functions, with a concave objective function, and a constraint that depends on the solution of the Skorokhod problem displayed in Theorem~\ref{samplepathldpofstochasticnetworks}. Though this Skorokhod problem only needs to be evaluated for step functions, this minimization problem is
in general not tractable. To get more concrete results we look at more specific functionals of the buffer content process in subsequent sections.
 
 \section{Asymptotics for overflow probabilities }\label{SFN-overflow-prob}
% \CR{We fix $\bm b$ and suppress the dependence in $\mathscr B$. On the other hand, we make the dependence on $\bm b$ explicit in $(P_{\bm b})$ and suppress the dependence in $y$. I replaced $(P_{\bm b})$ with $(P_{\bm b,y})$ to make the notations slightly more consistent.frrc
This section examines the probability that the buffer content associated with a subset of nodes in the system exceeds a high level.
In particular, we fix  \linkdest{nota-bm-b}$\bm{b}=(b_1,\ldots,b_{d}) \in \R_+^{d}$ and 
% consider a functional $\mathscr B: \prod_{i=1}^d\D[0,T] \to \R_+$ defined as
% \linkdest{nota-mathscr-B}$\mathscr{B}(\bm\eta)\triangleq \bm{b}^\intercal \pi(\phi(\bm\eta)) =  \bm{b}^\intercal \phi(\bm\eta)(T)$ for $\eta \in \prod_{i=1}^d\D[0,T]$.
% We 
study the probability of linear combination of the  buffer content at the end of the time horizon exceeding a threshold given by 
$\P(
\rvtxt{(1)-1}{(1)}{\bm{b}^\intercal \bm{Z}_n(T)} \geq y)$. 
%=\P(\mathscr B(\bm X_n) \geq y)$. 
Note that for the unscaled process $Z$, this is the probability of congestion at time $nT$.
Let  
% \M{change the notation for domain of $\xi$: the space for the LDP of Xn is $\prod_{i=1}^{d}\D^{-(\mathcal{Q}\bm{r})_i}[0,T]$. It is still ok to use a subscript $\leqslant \infty$ since this is space where $\tilde{I}$ is finite}
\begin{align*}
\linkdest{nota-I-prime}
I'(x) 
\triangleq 
\inf\Big\{\tilde{I}^{(d)}(\bm \xi):  
\rvtxt{(1)-2}{(1)}{
\bm b^\intercal \phi(\bm\xi)(T)
}
=x, \ \bm\xi \in \prod_{i=1}^{d}\D^{-(\mathcal{Q}\bm{r})_i}[0,T]\Big\}
% \\
% &
% =
% \inf\left\{\tilde{I}^{(d)}(\bm\xi):  \bm\xi\in \mathscr{B}^{-1}(x)\right\}.
\end{align*}
Define the set
$\linkdest{nota-V_geqslant}V_{\geqslant}(y) \triangleq \{\bm\xi \in \prod_{i=1}^{d} \D^{(\boldsymbol{\gamma}-\mathcal{Q}\bm{r})_i}_{\leqslant \infty}[0,T]: 
\rvtxt{(1)-3}{(1)}{
\bm b^\intercal \phi(\bm\xi)(T)
}
\geq y \},$ and let $V_{\geqslant}^*(y)$ be the optimal value of $\tilde I^{(d)}$ over the set $V_{\geqslant}(y)$; 
i.e.\ \linkdest{nota-V_geqslant^star}$V_{\geqslant}^*(y)\triangleq\inf_{\bm\xi \in V_{\geqslant}(y)}\tilde{I}^{(d)}(\bm\xi)$.
Similarly, let \linkdest{nota-V_>}$V_{>}(y) \triangleq \{\bm\xi \in  \prod_{i=1}^{d} \D^{(\boldsymbol{\gamma}-\mathcal{Q}\bm{r})_i}_{\leqslant \infty}[0,T]: 
\rvtxt{(1)-4}{(1)}{
\bm b^\intercal \phi(\bm\xi)(T)
}
> y \}$ and set  \linkdest{nota-V_>^*}$V_{>}^*(y)\triangleq\inf_{\bm\xi \in V_{>}(y)}\tilde{I}^{(d)}(\bm\xi).$
Note that
$V^*_{\geqslant}(y)$ and $V^*_{>}(y)$ depend on $T$, but we suppress the dependence for notational simplicity.

Recall that $\mathcal{J}$ is the set of nodes with exogenous input. 
Next, let \linkdest{nota-I^+}$I^+\triangleq \{j \in \{1,\ldots,d\}: b_j >0\}$.
The following two lemmas, proven in Section~\ref{proofs-stochastic-networks}, ensure the continuity of $\mathrm{V}^*_{\geqslant}(\cdot)$.
\begin{lemma}\label{convergence-of-quasi-variational-V}
\linktoproof{proof of convergence-of-quasi-variational-V}
Assume that $\mathcal{J} \cap I^+ \neq \emptyset$. The map $x\mapsto V_{\geqslant}^*(x)$ is $\alpha$-Hölder continuous:
$$|V^*_{\geqslant}(y) - V^*_{\geqslant}(x)| \leq \Big(\max_{i \in I^+}\frac{c_i}{b_i^{\alpha}}\Big) \cdot |y-x|^{\alpha}.$$
\end{lemma}

\begin{lemma}\label{equality-of-VbVbt}
\linktoproof{proof of equality-of-VbVbt}
Assume that $\mathcal{J} \cap I^+ \neq \emptyset$. It holds that $V^*_{\geqslant}(y)=V^*_{>}(y)$.
\end{lemma}

We are ready to prove the main result of this Section:
\begin{theorem}\label{LDP-Y-end-of-time-horizon}
For a fixed $\bm{b}=(b_1,\ldots,b_{d}) \in \R_+^{d} $ assume that $\mathcal{J} \cap I^+ \neq \emptyset$. 
The overflow probabilities  $ \P\left( 
\rvtxt{(1)-5}{(1)}
{\bm b^\intercal \bm{Z}_n(T)}
\geq y \right)$ satisfy 
the following logarithmic asymptotics:
\begin{equation}\label{asymptotics-overrflow-probabilities-ZnT}
\lim_{n \to \infty}\frac{1}{L(n)n^{\alpha}}\log \P(
\rvtxt{(1)-6}{(1)}
{\bm b^\intercal \bm{Z}_n(T)}
\geq y)=-V_{\geqslant}^*(y).
\end{equation}
\end{theorem}

\begin{proof}
Note first that from Lemma~\ref{continuity-of-the-projection-map}, 
\rvtxt{(1)-7}{(1)}
{$\bm b^\intercal \bm{Z}_n(T)$}
is a Lipschitz (w.r.t.\ $d_p$) image of $\bm Z_n$. 
Note also that $I'(y) = \inf\{I_{\bm Z}(\xi): \bm b^\intercal \xi(T) = y\}$. 
Therefore, applying Lemma~\ref{alternative-yes-and-not-extended-contraction-principle} \textit{i)} and Theorem~\ref{samplepathldpofstochasticnetworks}, we get the asymptotic upper and lower bounds for $\frac{1}{L(n)n^{\alpha}}\log\P(
\rvtxt{(1)-8}{(1)}
{\bm b^\intercal \bm{Z}_n(T)}
\geq y)$ as follows:
\[
\limsup_{n \to \infty}\frac{1}{L(n)n^{\alpha}}\log\P(
\rvtxt{3-1}{3}
{\bm b^\intercal \bm{Z}_n(T)}
\geq y)  \leq -\lim_{\epsilon \to 0}\inf_{x \in [y-\epsilon, \infty)}I'(x)
\]
and
\begin{align*}
\liminf_{n \to \infty}\frac{1}{L(n)n^{\alpha}}\log\P(
\rvtxt{3-2}{3}
{\bm b^\intercal \bm{Z}_n(T)}
\geq  y)
% &\geq
% \liminf_{n \to \infty}\frac{1}{L(n)n^{\alpha}}\log\P(\mathscr{B}(\bm{Z}_n)> y)
% \\&
\geq 
-\inf_{x \in (y, \infty)}I'(x).
\end{align*}
%  Therefore, 
%  $
%  \lim_{n \to \infty}\frac{1}{L(n)n^{\alpha}}\log\P(\mathscr{B}(\bm{Z}_n) \geq y) \leq -V^*_{\geqslant}(y)
%   $
%  which implies the desired upper bound.
However, from Lemma~\ref{convergence-of-quasi-variational-V} and Lemma~\ref{equality-of-VbVbt},
\begin{align*}
-\lim_{\epsilon \to 0}\inf_{x \in [y-\epsilon, \infty)}I'(x)
&
=
-\lim_{\epsilon \to 0}V^*_{\geqslant}\left(y-\epsilon \right)
= -V^*_{\geqslant}\left(y\right);
\\
-\inf_{x \in (y, \infty)}I'(x)
&
=
-V_{>}^*(y) 
= 
-V_{\geqslant}^*(y).
\end{align*}
That is, the upper and lower bounds for $\limsup$ and $\liminf$ match, and hence, the limit \eqref{asymptotics-overrflow-probabilities-ZnT} exists and equals $-V_{\geqslant}^*(y)$. 
\end{proof}
Note that $V_{\geqslant}^*(y)$ is the solution of an infinite dimensional optimization problem. 
%We will see in the next section that it is possible to reduce the optimization problem to a finite-dimensional problem in case of a two-node tandem network. 
We conjecture that in many problem instances, there exists a $k\geq 1$ (that depends on the specific network) such that $\prod_{i=1}^d \D_{\leqslant k}^{(\bm \gamma - \mathcal Q \bm r)_i}[0,T]$ contains an optimal path that minimizes the rate function $\tilde I(\cdot)$ over $V_{\geqslant}(y)$. 
In such cases, $V_\geqslant^*(y)$ can be computed by solving the following optimization problem.
For given $\bm b\in \R_+^d$ and $y>0$, let $P_{y,k}^*$ denote the optimal value of the following optimization problem:
\begin{equation}\label{Pb}\tag{$P_{y,k}$}
\begin{aligned} %\label{2d-V*c}
\inf\quad &\sum_{i=1}^{d} c_i\sum_{j=1}^k\left( x_i^{(j)}\right)^{\alpha} \nonumber 
\\
\text{s.t.} 
\quad 
&
\rvtxt{(1)-9}{(1)}
{\bm b^\intercal \phi(\bm\xi)(T)}
\geq y; 
\nonumber
\\
&
\xi_i = \textstyle\sum_{j=1}^kx_i^{(j)}\mathbbm{1}_{[u_i^{(j)},T]}+(\boldsymbol{\gamma}-\mathcal{Q}\bm{r})_1\cdot e;
% ,\ \ldots,\ \sum_{j=1}^k x_d^{(j)} \mathbbm{1}_{[u_d^{(j)},T]}+(\boldsymbol{\gamma}-\mathcal{Q}\bm{r})_d\cdot e
\nonumber
\\
&
x_i^{(j)} \geq 0 \ \text{for} \ i \in \mathcal{J},\ j \in \{1,\ldots,k\}, \quad \text{and}\quad x_i^{(j)}=0 \ \text{for} \ i \notin \mathcal{J},\ j \in \{1,\ldots,k\};
\nonumber
\\
&
u_i^{(j)} \in [0,T] \ \text{for} \ i\in\{1,\ldots,d\},\ j\in \{1,\ldots,k\}. \nonumber 
\end{aligned}
\end{equation}
%{\bf in the latex file this program has a label. But only 3rd and 4th line are numbered}
Then, $P^*_{y,k} = V^*_\geqslant(y)$.  
Note that this means that the large deviations rate is the solution of a $2kd$-dimensional optimization problem: the decision variables are the size $x_i^{(j)}$ and the time $u_i^{(j)}$ of the $k$ jumps ($j\in\{1,\ldots,k\}$) in the $d$ coordinates ($i\in\{1,\ldots,d\}$). 
This provides a significant reduction in complexity compared to the general setting of Section \ref{SFN-SP-BCP}. 
Nevertheless, even the finite dimensional problem \eqref{Pb} is still rather intricate: it is an $L^{\alpha}$-norm
minimization problem with $\alpha\in(0,1)$. In general, such problems are strongly NP-hard; see \cite{Ye2011}, for example. In addition, checking whether a solution to \eqref{Pb} is feasible requires one to compute the Skorokhod map $\phi$ for step functions, which is nontrivial.
To get more explicit results and gain some physical insights, we consider a two-node tandem network in the next section, where we can reduce the computation of $V_\geqslant^*(y)$ down to solving \eqref{Pb} with $k=1$. 
%In particular, we focus on a variation of the multiple on-off sources model. 

\section{A two-node example}
\label{tandem}

We consider a two-node tandem network where content from node 1 flows into node 2, and content from node 2 leaves the system, i.e.\ $q_{12}=1$, and $q_{ij}=0$ otherwise. 
We assume that each node has an exogenous input process (i.e.\ $\mathcal{ J} = \{1,2\}$). We consider the problem of identifying the log-asymptotics of the probability of congestion in the second node, i.e., $\P\left(\bm b^\intercal\bm{Z}_n(T) \geq y\right)$  as $n\to\infty$ where $\bm b = (0,1)$. 
That is, our goal is to compute $V_\geqslant^*(y)$ in this specific example.

The next lemma
%, which is proven in Section~\ref{proofs-stochastic-networks},
enables us to reduce the feasible region of the optimization problem associated with $V^*_\leqslant(y)$ from $\D^{(\boldsymbol{\gamma}-\mathcal{Q}\bm{r})_1}_{\leqslant \infty}[0,T]\times \D^{(\boldsymbol{\gamma}-\mathcal{Q}\bm{r})_2}_{\leqslant \infty}[0,T]$ down to 
$\D^{(\boldsymbol{\gamma}-\mathcal{Q}\bm{r})_1}_{\leqslant 1}[0,T]\times \D^{(\boldsymbol{\gamma}-\mathcal{Q}\bm{r})_2}_{\leqslant 1}[0,T]$.
In other words, we can restrict the class of functions to those that have at most one discontinuity in each coordinate.  

%\CR{$\prod_{i=1}^{d} \left(\D^{(\boldsymbol{\gamma}-\mathcal{Q}\bm{r})_i}[0,T] \cap \D^\uparrow_{\leqslant 1}[0,T]\right)$ and what we describe verbally don't match. Note that $\prod_{i=1}^{d} \left(\D^{(\boldsymbol{\gamma}-\mathcal{Q}\bm{r})_i}[0,T] \cap \D^\uparrow_{\leqslant 1}[0,T]\right) = \emptyset$ unless $(\boldsymbol{\gamma}-\mathcal{Q}\bm{r})_i = 0$ for all $i$'s.} 
%BZ: rewrote 
%\M{We insert the subscript $\leqslant \infty$. The following Lemma is needed to reduce the effective domain of the variational problem associated with LDP to one step functions. }
\begin{lemma}\label{reduction-to-one-step-functions}
\linktoproof{proof of reduction-to-one-step-functions}
Consider the two-node tandem network where $d = 2$ and $\mathcal Q = \begin{pmatrix}
1&0\\-1&1
\end{pmatrix}$. 
Let $\bm\xi \in \prod_{i=1}^d\D^{(\boldsymbol{\gamma}-\mathcal{Q}\bm{r})_i}_{\leqslant \infty}[0,T]$. 
Then, there exists a path $\tilde{\bm\xi}\in  
\prod_{i=1}^d\D^{(\boldsymbol{\gamma}-\mathcal{Q}\bm{r})_i}_{\leqslant 1}[0,T]$ such that
\begin{itemize}
\item[i)] $\tilde{I}^{(d)}(\tilde{\bm\xi}) \leq \tilde{I}^{(d)}(\bm\xi)$,
\item[ii)] $\phi(\tilde{\bm\xi})(T) \geq \phi(\bm\xi)(T)$.
 \end{itemize}
\end{lemma}

Lemma~\ref{reduction-to-one-step-functions} implies that computing $V_\geqslant^*(y)$ is equivalent to solving \eqref{Pb} with $k=1$ in case of the two-node tandem networks.
% and $\bm b = (0,1)$, 
Such computation is the subject of the rest of this section. To keep the presentation concise, we  give an outline of the key steps and focus on physical insight.

We first develop an explicit expression for the buffer content at time $T$ for input processes of the form $\xi_i = (\boldsymbol{\gamma} -\mathcal{Q}\bm{r})_i \cdot e + x_i \mathbbm{1}_{[u_i,T]}, \ t\in [0,T]$, \ $x_{i} \geq 0$, \ $u_i\in [0,T], \ i=1,2$.
To develop physical intuition is it instructive to write the buffer content process at node 2 as the solution of a one-dimensional reflection mapping, fed by the superposition of $\xi_2$ and
the output process of node $1$, which in turn is governed by a one-dimensional reflection mapping as well.
\CR{\\
\[
Q = \begin{pmatrix}
0 & 1 \\ 0 &0
\end{pmatrix}
\quad
\mathcal Q = (I-Q^\intercal) = \begin{pmatrix}
1&0\\-1&1
\end{pmatrix}
\quad
\bm \xi + \mathcal Q \psi(\bm\xi) = \begin{pmatrix}
\xi_1 + \psi_1(\bm\xi)
\\
\xi_2 - \psi_1(\bm\xi) + \psi_2(\bm \xi)
\end{pmatrix}
\]
\[
\psi_1(\bm\xi)(t) = - \inf_{s\leq t}\{ \xi_1(s)\} 
\qquad\implies\qquad 
\phi_1(\bm\xi)(t) = \xi_1(t) - \inf_{s\leq t}\{ \xi_1(s)\}
\] 
\[
(\bm \xi + \mathcal Q \psi(\bm\xi))_2(t) 
=
(\xi_2 - \psi_1(\bm\xi) + \psi_2(\bm \xi))(t)
=
\xi_2(t) + \inf_{s\leq t}\{ \xi_1(s)\} + \psi_2(\bm \xi)(t)
\]
\[
\psi_2(\bm\xi) = -\inf_{s\leq t}\{ \xi_2(s) + \inf_{u\leq s}\{ \xi_1(u)\} \}
\]
\[
\phi_2(\bm\xi)(T)= \xi_2(T)+ \inf_{s\leq T} \xi_1(s) - \inf_{u\leq T}\Big\{\xi_2(u)+ \inf_{s\leq u} \xi_1(s)\Big\}.
\]
}%
To this end, observe that $\psi_1(\bm\xi)(t) = - \inf_{s\leq t}\{0\wedge \xi_1(s)\}$, and $\phi_1(\bm\xi)(t) = \xi_1(t) - \inf_{s\leq t}\{ 0\wedge\xi_1(s)\}$. 
Note also that
$(\bm \xi + \mathcal Q \psi(\bm\xi))_2 
=
\xi_2 - \psi_1(\bm\xi) + \psi_2(\bm \xi)
$, and the minimal $\psi_2(\bm\xi)$ that regulates this process above zero is 
$
\psi_2(\bm\xi)(t)= 
-\inf_{s\leq t}\{ 0\wedge (\xi_2(s) + \inf_{u\leq s}\{ 0\wedge\xi_1(u)\}) \}
$.
% The amount of fluid flowing from node $1$ to node $2$ in $[0,t]$ is given by $\xi_1(t)-z_1(t)=- \inf_{s\leq t} \xi_1(s)$.
Consequently, we can write 
\begin{equation}
\label{skorokhod-node2}
    \phi_2(\bm\xi)(T)= \xi_2(T)+ \inf_{s\leq T} \{0\wedge\xi_1(s)\} - \inf_{u\leq T}\Big\{0\wedge\big\{\xi_2(u)+ \inf_{s\leq u} \{0\wedge\xi_1(s)\}\big\}\Big\}.
\end{equation}
Our goal is to minimize the cost $c_1x_1^\alpha+c_2 x_2^\alpha$ subject to the constraint $\phi_2(\bm\xi)(T)\geq y$, over $x_1 \geq 0,\ x_2 \geq 0,\ u_1 \in [0,T],\ u_2\in [0,T]$. We simplify this problem by identifying convenient choices of $u_1$ and $u_2$ which do not lose optimality.

To this end, observe that a jump of size $x_2$ at time $u_2$ can instead take place at time $u_2=T$ without decreasing $\phi_2(\bm\xi)(T)$. 
To determine a convenient choice of $u_1$, note that a jump of size $x_1$ in node 1 at time $u_1$ causes an outflow of rate $r_1$ from node $1$ to node $2$ in the interval $[u_1, u_1 + x_1/(r_1-\gamma_1)]$, and rate $\gamma_1$ after time 
$u_1 + x_1/(r_1-\gamma_1)$. Therefore, we can take $u_1$ such that $u_1 + x_1/(r_1-\gamma_1)=T$, without decreasing $\phi_2(\bm\xi)(T)$. This choice is feasible as long as $u_1$ remains non-negative, i.e.\ we require that $x_1/(r_1-\gamma_1) \leq T$. Observe that choosing $x_1/(r_1-\gamma_1)>T$ would not be optimal, as it would increase the cost term involving $x_1^\alpha$ without increasing $\phi_2(\bm\xi)(T)$.

We proceed by solving (\ref{skorokhod-node2}) by taking $\xi_1 = (\boldsymbol{\gamma} -\mathcal{Q}\bm{r})_1 \cdot e + x_1 \mathbbm{1}_{[T- x_1/(r_1-\gamma_1),T]}$ and $\xi_2= (\boldsymbol{\gamma} -\mathcal{Q}\bm{r})_2 \cdot e + x_2\mathbbm{1}_{[T,T]}$. 
Straightforward manipulations show that
\begin{equation}
    \phi_2(\bm\xi)(T) = x_2 + (r_1+\gamma_2-r_2)^+ \frac{x_1}{r_1-\gamma_1}.
\end{equation}
We see that a jump at node 1 has no effect on the buffer content in node 2 if $r_2 \geq r_1+\gamma_2$, which is intuitively obvious since node 2 is still rate stable when the output of node $1$ equals $r_1$. 
Therefore, $x_1=0$ and $x_2 = y$ is feasible and minimizes the rate function. 
Our first conclusion is that
\begin{equation}
\lim_{n \to \infty}\frac{1}{L(n)n^{\alpha}}{\log \P\left(\bm b^\intercal\bm{Z}_n(T) \geq y\right)}=-c_2y^{\alpha}, \hspace{1cm} r_2 \geq r_1+\gamma_2.
\end{equation}
We now turn to the more interesting case $r_2 < r_1+\gamma_2$. We do not lose optimality if the constraint on $\phi_2(\bm\xi)(T)$ is tight, so we can impose the constraints
\begin{equation}
\label{polyhedralset}
    x_2 + \frac{r_1+\gamma_2-r_2}{r_1-\gamma_1} x_1 = y, \hspace{0.5cm} x_1 \in [0, (r_1-\gamma_1) T], \hspace{0.5cm} x_2 \geq 0.
\end{equation}
From convex optimization theory, see Corollary~32.3.2 in \cite{rockafellar1970convex}, the minimum of the concave objective function
 $c_1x_1^\alpha+c_2 x_2^\alpha$ subject to the constraints (\ref{polyhedralset}) 
 is achieved over the extreme points of (\ref{polyhedralset}). In our particular situation, this implies that an optimal solution should correspond to one of the following 3 cases: (i)  $x_1=0$, (ii) $x_2=0$, (iii) $x_1= (r_1-\gamma_1) T$. 
In case $(iii)$ we would have $x_2=y-(r_1+\gamma_2-r_2)T$, which is only feasible if $y\geq (r_1+\gamma_2-r_2)T$.  
Note also that if $y = (r_1+\gamma_2-r_2)T$, then (ii) is the case. 
  
Therefore, if $y\leq (r_1+\gamma_2-r_2)T$, we can conclude that either case (i) holds with $x_1=0, x_2=y$, and cost $c_2 y^\alpha$, or case (ii) holds with $x_2=0$, 
$x_1= y\frac{r_1-\gamma_1}{r_1+\gamma_2-r_2}$, and cost $c_1 \left( y\frac{r_1-\gamma_1}{r_1+\gamma_2-r_2}\right)^\alpha$. We conclude that for $y\leq  (r_1+\gamma_2-r_2)T$,
 \begin{equation}
\lim_{n \to \infty}\frac{1}{L(n)n^{\alpha}}{\log \P\left(\bm b^\intercal \bm{Z}_n(T) \geq y\right)}=
-\min \left\{c_1 \left(\frac{r_1-\gamma_1}{r_1+\gamma_2-r_2}\right)^\alpha,  c_2\right\}y^\alpha.
\end{equation}
We now turn to the case $y> (r_1+\gamma_2-r_2)T$. 
In this case, the time horizon $T$ is small w.r.t.\ $y$:  the output of node $1$ alone is never enough to cause the buffer content of node $2$ to reach level $y$ at time $T$. 
Thus, case (ii) can be excluded, and we only have to compare case (i) and case (iii).  Case (i) has solution $x_2=y$ with cost $c_2 y^\alpha$. Case (iii) has solution $x_1= (r_1-\gamma_1) T$, $x_2=y-(r_1+\gamma_2-r_2)T$,
with cost $c_1 ((r_1-\gamma_1) T)^\alpha   + c_2 (y-(r_1+\gamma_2-r_2)T)^\alpha$. We conclude that, if 
$y> (r_1+\gamma_2-r_2)T$,
 \begin{equation}
 \label{examplehardestcase}
\lim_{n \to \infty}\frac{1}{L(n)n^{\alpha}}{\log \P\left(\bm b^\intercal \bm{Z}_n(T) \geq y\right)}=
-\min \{ c_2 y^\alpha ,  c_1 ((r_1-\gamma_1) T)^\alpha + c_2 (y-(r_1+\gamma_2-r_2)T)^\alpha  \}. 
\end{equation}
To give a numerical example, take $y=2, T=1, r_1=r_2=3, \gamma_1=\gamma_2=1$. In this case, the inequality  $y> (r_1+\gamma_2-r_2)T$ holds. To evaluate (\ref{examplehardestcase}), note that the cost of case (i) equals $c_2 2^{\alpha}$ and the cost of case (iii) equals $c_1 2^\alpha + c_2$. 
So we conclude that case (iii) is the most likely way for the event  $\{\bm b^\intercal \bm{Z}_n(1) \geq 2\}$ to occur if $c_1 \leq c_2 (1-2^{-\alpha})$, corresponding to a most likely behavior of two big jumps: $x_1=2$, occuring at node 1 at time 0, and $x_2=1$, occuring at node 2 at time 1.

One may wonder if Lemma~\ref{reduction-to-one-step-functions} can be extended to general stochastic fluid networks so that the computation of $V_\geqslant^*(y)$ can always be reduced to solving $\eqref{Pb}$ with $k=1$.
%for general stochastic fluid networks. 
(This means that their large deviations behaviors are consequences of at most one jump in the external input process to each node.)
Unfortunately, this is not the case. 
We conclude this section with an example for which restricting the number of jumps in each coordinate to at most one is strictly sub-optimal. 
% \CR{TODO: polish the example}

% Given such an example, a natural question is that reduction to \eqref{Pb} with $k=1$ is possible for general stochastic fluid networks in view of the ``principle of a single big jump.''
% However, reduction of $V_\geqslant^*(y)$ to \eqref{Pb} with $k=1$ is not always possible. 

Consider $\alpha = 1/2$, $T = 2$, $y = 2 + \delta\theta$,
$$
\bm{\gamma} 
= 
\begin{pmatrix}
\epsilon\\
0\\
0\\
\end{pmatrix},
\qquad
\bm{r} 
= 
\begin{pmatrix}
4+\epsilon\\
2+\epsilon\\
1+\epsilon\\
\end{pmatrix},
\qquad
Q = 
\begin{pmatrix}
0 & 1 & 0 \\
0 & 0 & 1 \\
0 & 0 & 0 \\
\end{pmatrix},
\qquad
\bm{b} 
= 
\begin{pmatrix}
\delta\\
0\\
1\\
\end{pmatrix},
\qquad
\bm{\gamma} - \mathcal Q \bm{r} 
= 
\begin{pmatrix}
-4\\
2\\
1\\
\end{pmatrix},
$$
where $\epsilon = 0.1$, $\delta < 1/4$, $\theta < 1$, and $c_1 = c_2 = 1$.
Let $\bm\xi$ be the superposition of the fluid limit $(\bm{\gamma} - \mathcal Q\bm{r})\cdot e$ of the potential buffer content vector and two jumps of size 4 and $\theta$ in the first coordinate at the beginning and at the end of the time horizon, respectively. That is, 
$$
\bm\xi(t)
= 
\begin{pmatrix}
-4t + 4\ind_{[0,T]}(t) + \theta \ind_{[T,T]}(t)\\
2t\\
t\\
\end{pmatrix}
.
$$
Then, $\tilde I^{(d)}(\bm\xi) = 2 + \sqrt{\theta}$ and 
$$
\phi(\bm\xi)(T) = 
\begin{pmatrix}
\theta\\
0\\
2\\
\end{pmatrix}
.
$$
However, any $\tilde{\bm\xi}$ (in the effective domain of $\tilde I^{(d)}$) with only one jump in the first coordinate takes the following form:
$$
\tilde{\bm\xi}(t)
= 
\begin{pmatrix}
-4t + x\ind_{[s,T]}(t)\\
2t\\
t\\
\end{pmatrix}
$$
for some $s\in[0,T]$ and $x\in(0,\infty)$. 
Note that if $s>0$, 
\rvtxt{4}{4}
{the third coordinate} 
cannot reach 2. 
Therefore, we see that $s$ has to be zero. 
Now, we see that for $\phi(\tilde{\bm\xi})(T)$ to be greater than $\phi(\bm\xi)(T)$ cooridnate-wise as claimed in ii) of Lemma~\ref{reduction-to-one-step-functions}, $x$ has to be at least $4T + \theta$. 
However, since $\delta < 1$, this means that $\tilde I^{(d)}(\tilde{\bm\xi}) \geq \sqrt{4T + \theta} > \sqrt 4 + \sqrt{\theta} = \tilde I^{(d)}(\bm\xi)$. That is, no $\tilde {\bm\xi}$ with only one jump in the first coordinate satisfies the conclusion of Lemma~\ref{reduction-to-one-step-functions}.
In fact, this system of tandem queues still turns out to be a counterexample even if we change the statement of Lemma~\ref{reduction-to-one-step-functions} so that $ii)$ is  $\bm{b}^\intercal \phi(\tilde\xi)(T) \geq \bm{b}^\intercal \phi(\xi)(T)$. 
To see this, note first that if $x < 4(T-s)$, then $\bm{b}^\intercal\phi(\tilde{\bm\xi})(T) < y$, and hence, we only consider the case $x \geq 4(T-s)$, where
$$
\bm{b}^\intercal \phi(\tilde{\bm\xi})(T) = \delta( x - 4(T-s)) + T-s = \delta x + (1-4\delta)(T-s).
$$
Note also that since we assume $\delta < 1/4$, this is maximized at $s = 0$. 
Therefore, for $\bm{b}^\intercal \phi(\tilde{\bm\xi})(T)$ to be greater than or equal to $y$, we need $x$ to be greater than or equal to $4T + \theta$. 
This implies that $\tilde I^{(d)} (\tilde {\bm\xi}) \geq \sqrt{4T + \theta} > \sqrt 4 + \sqrt \theta = \tilde I^{(d)}({\bm\xi})$.
Therefore, solving $\eqref{Pb}$ with $k=1$ won't give the correct log asymptotics for $\P(\bm{b}^\intercal \phi(\bm X_n)(T)\geq y)$ in general.

\section{Complementary proofs}\label{proofs-stochastic-networks}
% This section contains proofs for key results used in the main body of this paper.
 
 \subsection{Proofs of Lemma~\ref{convergence-of-quasi-variational-V} and \ref{equality-of-VbVbt}}
 
%We start with Lemma~\ref{reduction-to-one-step-functions}. 

Next, we focus on the continuity of $\mathrm{V}^*_{\geqslant}(\cdot)$. Let \linkdest{nota-D-+}$\D_+[0,T]$ be the subspace of $\D[0,T]$ that contains paths with only positive discontinuities: $\D_+[0,T] = \{\xi \in \D[0,T]: \xi(t)-\xi(t-) \geq 0,\ \forall t\in [0,T]\}$. 
Recall that $\D_{\leqslant k}[0,T] = \{\xi\in \D[0,T]: |{Disc}(\xi)| \leq k\}$.

\begin{lemma}\label{construction-of-extra-jump}
Suppose that $\bm a=(a_1,\ldots,a_d) \in \R^{d}_+$, 
% $\bm b = (b_1,\ldots,b_d) \in \R^d$,
$\bm \xi \in \prod_{i=1}^{d}\D_{\leqslant \infty}^{(\bm \gamma - \mathcal Q \bm r)_i}[0,T]$, and $\bm\zeta = \bm\xi + \bm a\mathbbm{1}_{\{T\}}$. 
Then 
\begin{itemize}
\item[i)]  $\psi(\bm\zeta)=\psi(\bm\xi)$, 
\item[ii)]  $\phi(\bm\zeta)(T)=\phi(\bm\xi)(T)+a$, and 
\item[iii)] $\tilde{I}^{(d)}(\bm\zeta) \leq \tilde{I}^{(d)}(\bm\xi)+\sum_{i=1}^{d}c_ia_i^{\alpha}.$
\end{itemize} 
\end{lemma}

\begin{proof}
For \textit{i)}, from the proof of the Theorem~14.2.2 in \cite{whitt2002stochastic}, we see that for any $\bm\omega \in \prod_{i=1}^{k}\D[0,T]$ the regulator component $\psi(\bm\omega)$ is the limit (w.r.t.\ $\|\cdot\|$) of $\rho_{\bm\omega}^n(\bm 0)$ where $\bm 0$ is the zero function and $\rho_{\bm\omega}^n$ is the $n$ fold composition of $\rho_{\bm\omega}:\prod_{i=1}^d\D^\uparrow[0,T]\to \prod_{i=1}^d\D^\uparrow[0,T]$ such that $\rho_{\bm\omega}(\bm\eta)(t) = 0 \vee \sup_{s\in[0,t]} \{Q\bm\eta(s) - \bm\omega(s)\}$.
Note that $\rho_{\bm\omega}(\bm\eta)(t)$ depends only on $\bm\eta(s)$ and $\bm\omega(s)$ for $s\in[0,t]$. 
Therefore, $\psi(\bm\omega)(t)$ depends on $\bm\omega(s)$ for $s\in[0,t]$ only.  
Therefore, $\psi(\bm\zeta)(t) = \psi(\bm\xi)(t)$ for $t\in [0,T-\epsilon]$ for any $\epsilon>0$. 
The continuity implies that $\psi(\bm\zeta)(T) = \psi(\bm\xi)(T)$ as well, which concludes the proof of part i). 
\iffalse
We first claim that $\psi(\zeta)(t) = \psi(\xi)(t)$ for $t\in[0,T-\epsilon]$ for any fixed $\epsilon>0$. 
Obviously, $\rho^1_\zeta(\bm 0)(t)= 0\vee\sup_{s \in[0, t]}\{-\zeta(s)\},0\}= 0\vee\sup_{s \in[0, t]}\{-\xi(s)\},0\} = \rho^1_\xi(\bm 0)(t)$  for $ t \leq T-\epsilon$. 
To proceed with induction, suppose that $\rho_\zeta^k(\bm 0)(t) = \rho_\xi^k(\bm 0)(t)$ for $t\in[0,T-\epsilon]$. Then,
 for any $t \in[0, T-\epsilon]$ we have
 \[
\rho^{k+1}_\zeta(\bm 0)(t)
 =
 0\vee\sup_{s \in[0, t]}\left\{Q\rho^k_\zeta(\bm 0)(s) -\zeta(s)\right\}
 =
 0\vee\sup_{s \in[0, t]}\left\{Q\rho^k_\xi(\bm 0)(s) -\xi(s)\right\} 
 = 
 \rho^{k+1}_\xi(\bm 0)(t).
 \]
Since the equality holds for every $k \in \N$, their (uniform) limits $\psi(\zeta)$ and $\psi(\xi)$ should also coincide on $[0,T-\epsilon]$, proving the claim. 
Since $\epsilon$ was arbitrary, this means that $\psi(\zeta)(t) = \psi(\xi)(t)$, $t \in [0,T)$. 
Furthermore, since $\psi(\zeta)$ and $\psi(\xi)$ are continuous from Result~\ref{size-discontinuities-of-phi}, we conclude that they coincide on $[0,T]$, i.e., $\psi(\xi)=\psi(\zeta)$. 
\fi

For \textit{ii)}, observe that 
$
\phi(\bm\zeta)(T)= \bm\zeta(T)+\mathcal Q \psi(\bm\zeta)(T)=\bm\xi(T)+\bm a+\mathcal Q 
\psi(\bm\xi)(T)=\phi(\bm\xi)(T)+\bm a.
$

For \textit{iii)}, we assume that 
$\xi^{(j)}(t) = -(\mathcal{Q}\bm{r})_{j}(t) %\\ & 
\ \text{for} \ j \notin \mathcal{J}$ since if not $\tilde I^{(d)} (\bm\zeta) = \tilde I^{(d)}(\bm\xi) = \infty$, and the inequality holds trivially.
Let $\bm\zeta=(\zeta_1,\ldots,\zeta_d)$, and $\bm\xi=(\xi_1,\ldots,\xi_d)$. 
Since the function $x \mapsto x^{\alpha},  \ \alpha \in (0,1)$, is sub-additive, 
\begin{align*}
 I(\zeta_i)
% &
%  =
%  \sum_{t \in [0,T]: \zeta^{(i)}(t) \neq \zeta^{(i)}(t-) }\left(\zeta^{(i)}(t) - \zeta^{(i)}(t-)\right)^{\alpha}
%   \\
&
=
\sum_{t \in [0,T): \xi_i(t) \neq \xi_i(t-) }\left(\xi_i(t) - \xi_i(t-)\right)^{\alpha}
+\left(\xi_i(T)-\xi_i(T-)+a_i\right)^{\alpha}
\\
&
\leq
 \sum_{t \in [0,T): \xi_i(t) \neq \xi_i(t-) }\left(\xi_i(t) - \xi_i(t-)\right)^{\alpha}
 +\left(\xi_i(T)-\xi_i(T-)\right)^{\alpha}+a_i^{\alpha}
 \\
 &
 =
 I(\xi_i) +a_i^{\alpha}.
\end{align*}
Therefore, $\tilde{I}^{(d)}(\bm\zeta)
=
 \sum_{j\in \mathcal J}c_j I(\zeta_j)
  \leq
    \sum_{j\in\mathcal J} c_jI(\xi_j)+\sum_{j\in \mathcal J}c_ja_j^{\alpha}
    \leq \tilde{I}^{(d)}(\bm\xi)+ \sum_{j=1}^{d}c_ja_j^{\alpha}.$
\end{proof}
 
\linkdest{proof of convergence-of-quasi-variational-V}
\begin{proof}[Proof of Lemma~\ref{convergence-of-quasi-variational-V}]
W.l.o.g., let $y \geq x \geq 0$. 
Then $V_{\geqslant}(y) \subseteq V_{\geqslant}(x)$, and hence,
$V^*_{\geqslant}(y) \geq V^*_{\geqslant}(x) \geq 0.$   
For any $\epsilon>0$, there exists a $\bm\zeta \in V_{\geqslant}(x)$ so that $\tilde{I}^{(d)}(\bm\zeta)< V^*_{\geqslant}(x)+\epsilon$. 
Next, fix $j \in I^+$ and let $\bm\xi = \bm\zeta + \bm{v}\mathbbm{1}_{\{T\}}$ where $\bm{v}
= (0,\ldots,\frac{y-x}{b_j},\ldots,0)$. Due to \textit{ii)} of Lemma~\ref{construction-of-extra-jump}, $$
\bm{b}^\intercal\phi(\bm\xi)(T)
=
\bm{b}^\intercal(\phi(\bm\zeta)(T)+\bm{v}) 
\rvtxt{5}{5}{\ =\ }
\bm{b}^\intercal\phi(\bm\zeta)(T)+b_j\frac{(y-x)}{b_j} 
\geq 
x + y-x
=
y.
$$ Hence,  $\bm\xi \in V_{\geqslant}(y)$.
Due to \textit{iii)} of Lemma~\ref{construction-of-extra-jump}, 
$$
\tilde{I}^{(d)}(\bm\xi) 
\leq 
\tilde{I}^{(d)}(\bm\zeta)+\frac{c_j}{b_j^\alpha}\cdot(y-x)^{\alpha}
\leq 
\tilde{I}^{(d)}(\bm\zeta)+\Big(\max_{i \in I^+}\frac{c_i}{b_i^\alpha}\Big)\cdot(y-x)^{\alpha}.
$$ 
We see that
\[
V^*_{\geqslant}(y) \leq \tilde{I}^{(d)}(\bm\xi) \leq \tilde{I}^{(d)}(\bm\zeta)+\max_{1 \leq i \leq d: b_i >0}\frac{c_i}{b_i^\alpha}(y-x)^{\alpha} 
< V^*_{\geqslant}(x)+\max_{\{1 \leq i \leq d: b_i >0\}}\frac{c_i}{b_i^\alpha}(y-x)^{\alpha}+\epsilon.
\]
This leads to $V^*_{\geqslant}(y)-V^*_{\geqslant}(x) \leq \max_{\{1 \leq i \leq d: b_i >0\}}\frac{c_i}{b_i^\alpha}(y-x)^{\alpha}+\epsilon$.  We obtain the desired result by letting $\epsilon$ tend to $0$.
Thus, $|V^*_{\geqslant}(y) - V^*_{\geqslant}(x)| \leq \max_{\{1 \leq i \leq d: b_i >0\}}\frac{c_i}{b_i^\alpha} \cdot |y-x|^{\alpha}.$
\end{proof} 

We conclude this section with the proof of Lemma~\ref{equality-of-VbVbt}. 
%$V^*_{\geqslant}(y) = V^*_{>}(y)$.

\linkdest{proof of equality-of-VbVbt}
\begin{proof}[Proof of Lemma~\ref{equality-of-VbVbt}]
For any $\epsilon>0$, we have that 
 $V^*_{\geqslant}(y+\epsilon) \geq V^*_{>}(y)$.
Hence, in view of Lemma~\ref{convergence-of-quasi-variational-V},
\[
|V^*_{>}(y)-V^*_{\geqslant}(y)| 
=
V^*_{>}(y)-V^*_{\geqslant}(y)
\leq V^*_{\geqslant}(y+\epsilon)-V^*_{\geqslant}(y)
\leq 
\big(\max_{i \in I^+}\frac{c_i}{b_i^\alpha}\big)\cdot|\epsilon|^{\alpha}.
\]
Now, we let $\epsilon$ go to $0$ to obtain the desired result.
%It is obvious that $V^*_{>}(c) \geq V^*_{\geqslant}(c)$. Now, we prove that 
%$$V^*_{\geqslant}(c)=\inf_{\xi \in V_{\geqslant}(c)}\tilde{I}^{(d)}(\xi) \geq V^*_{>}(c).$$
%We only need to show that for  every $\epsilon>0$ there exists a $\xi \in V_{\geqslant}(c)$ so that $\tilde{I}^{(d)}(\xi)< V^*_{>}(c)+\epsilon.$
%By the definition of the infimum, let $\zeta \in V_{>}(c)$ be such that $\tilde{I}^{(d)}(\zeta) < V_{>}^*(c)+ \epsilon/2$. Choose an $i$ so that $b_i>0$; then, due to, Lemma~\ref{construction-of-extra-jump}, the path $\xi =\zeta+(0\ldots,(\frac{\epsilon}{2c_i})^{\frac{1}{\alpha}},\ldots,0)\mathbbm{1}_{\{T\}}$ belongs to $V_{\geqslant}(c)$. 
%To see this, $\bm{b}^\intercal\phi(\xi)(T)=\bm{b}^\intercal(\phi(\zeta)(T)+ \bm{b}^\intercal(0\ldots,\frac{\epsilon}{2b_i^{\alpha}},\ldots,0)
%\geq \bm{b}^\intercal\phi(\zeta)(T)+b_i(\frac{\epsilon}{2c_i})^{\frac{1}{\alpha}} \geq c$.
%Furthermore, from \textit{iii)} of lemma ~\ref{construction-of-extra-jump} we have that $\tilde{I}^{(d)}(\xi) \leq \tilde{I}^{(d)}(\zeta)+\epsilon/2$. Hence, 
%$\tilde{I}^{(d)}(\xi) < V^*_{>}(c)+\epsilon$. Since this construction holds true for any $\epsilon>0$, we have that 
%$$V^*_{\geqslant}(c)=\inf_{\xi \in V_{\geqslant}(c)}\tilde{I}^{(d)}(\xi) \geq V^*_{>}(c).$$
\end{proof}

\subsection{Proof of Lemma~\ref{reduction-to-one-step-functions}}
For any $\eta\in\D[0,T]$, let $\eta^\downarrow \in \D[0,T]$ denote the running infimum \linkdest{nota-one-dimensional-regulator}$\eta^\downarrow (t) \triangleq \inf_{s\in[0,t]}0\wedge \eta(s)$ for all $t\in[0,T]$. 
The following simple lemma is useful for proving Lemma~\ref{reduction-to-one-step-functions}.

\begin{lemma}\label{comparison of one dimensional reflection}
Suppose that $\eta, \omega \in \D[0,T]$ are such that $\eta \geq \omega$ and $\eta(T) = \omega(T)$. 
Then $(\eta -\eta^\downarrow)(T) \leq (\omega -\omega^\downarrow)(T)$.
\end{lemma}
\begin{proof}
Since $\eta^\downarrow \geq \omega^\downarrow$, we have
$
\eta - \eta^\downarrow \leq \eta - \omega^\downarrow. 
$
Therefore, 
$(\eta - \eta^\downarrow)(T) \leq (\eta - \omega^\downarrow)(T)  = (\omega - \omega^\downarrow)(T)$
\end{proof}

Now we prove Lemma~\ref{reduction-to-one-step-functions}. 
\linkdest{proof of reduction-to-one-step-functions}
\begin{proof}[Proof of Lemma~\ref{reduction-to-one-step-functions}]
Since we assume that $\bm \xi = (\xi_1,\xi_2) \in \D^{(\boldsymbol{\gamma}-\mathcal{Q}\bm{r})_1}_{\leqslant \infty}[0,T] \times \D^{(\boldsymbol{\gamma}-\mathcal{Q}\bm{r})_2}_{\leqslant \infty}[0,T]$, 
we can write, 
\rvtxt{(2)-3}{(2)}{%
$\xi_1 = (\gamma_1 - r_1)e + \sum_{j=1}^\infty x^{(j)} \one_{[u^{(j)},T]}$ and $\xi_2 = (\gamma_2 + r_1 - r_2)e + \sum_{j=1}^\infty y^{(j)} \one_{[v^{(j)},T]}$ for $x^{(j)}, y^{(j)}\geq 0$ and $u^{(j)},v^{(j)} \in [0,T]$%
}%
, $j=1,2,\ldots$. 
% We can assume w.l.o.g.\ that $\{u_j: j \geq 1\}$'s are all distinct. 
% Likewise, we can write, $\xi_1 = \sum_{j=1}^\infty x_j \one_{[u_j,T]}$ for $x_j\geq 0$ and $u_j \in [0,T]$, $j=1,2,\ldots$. 
Consider $\bm\xi' = (\xi_1, \xi_2')$ where $\xi_2' = (\gamma_2 + r_1 - r_2)e + \bar y \one_{[T,T]}$ and 
\rvtxt{(2)-4}{(2)}{%
$\bar y = \sum_{j=1}^\infty y^{(j)}$%
}. 
Then, by the subadditivity of $x\mapsto x^\alpha$, $\tilde I^{(d)}(\bm \xi') \leq \tilde I^{(d)}(\bm\xi)$.

Note that since 
$$
\bm \xi + \mathcal Q \psi(\bm\xi) = \begin{pmatrix}
\xi_1 + \psi_1(\bm\xi)
\\
\xi_2 - \psi_1(\bm\xi) + \psi_2(\bm \xi)
\end{pmatrix}
\qquad \text{and} \qquad
\bm \xi' + \mathcal Q \psi(\bm\xi') = \begin{pmatrix}
\xi_1 + \psi_1(\bm\xi')
\\
\xi_2' - \psi_1(\bm\xi') + \psi_2(\bm \xi')
\end{pmatrix},
$$
we see that $\psi_1(\bm\xi) = \psi_1(\bm\xi')= - \xi_1^\downarrow$, and hence, 
$\phi_1(\bm\xi) = \phi_1(\bm\xi') = \xi_1 - \xi_1^\downarrow$.
Also, 
$\phi_2(\bm\xi) = \xi_2 -\psi_1(\bm\xi) -  (\xi_2 - \psi_1(\bm\xi))^\downarrow$
and 
$\phi_2(\bm\xi') = \xi_2' -\psi_1(\bm\xi) -  (\xi_2' - \psi_1(\bm\xi))^\downarrow$.
% $\phi_2(\bm\xi)(t) = \xi_2(t) - \psi_1(\bm\xi)(t) - \inf_{s\leq t} 0\wedge \big(\xi_2(s) - \psi_1(\bm\xi')(s)\big)$ and 
% $\phi_2(\bm\xi')(t) = \xi_2'(t) - \psi_1(\bm\xi')(t) - \inf_{s\leq t} 0\wedge \big(\xi_2'(s) - \psi_1(\bm\xi')(s)\big)$. 
Note that since $\xi_2 - \psi_1(\bm\xi) \geq \xi_2' - \psi_1(\bm\xi)$ and $(\xi_2 - \psi_1(\bm\xi))(T) = (\xi_2' - \psi_1(\bm\xi))(T)$, 
Lemma~\ref{comparison of one dimensional reflection} implies that 
$$
\phi_2(\bm\xi)(T) = \xi_2 - \psi_1(\bm\xi) - (\xi_2 - \psi_1(\bm\xi))^\downarrow 
\leq \xi_2 - \psi_1(\bm\xi) - (\xi_2' - \psi_1(\bm\xi))^\downarrow 
=\phi_2(\bm\xi')(T).
$$
Therefore, we found $\bm\xi'\in \D_{\leqslant \infty}^{(\bm\gamma-\mathcal Q \bm r)_1}[0,T] \times \D_{\leqslant 1}^{(\bm\gamma-\mathcal Q \bm r)_2}[0,T] $ such that $\tilde I^{(d)}( \bm\xi') \leq \tilde I^{(d)}(\bm\xi)$ and $\phi(\bm\xi')(T) \geq \phi(\bm\xi)(T)$.
Now, 
let $\bm\xi'' \triangleq (\xi_1', \xi_2')$ where $\xi_1' = (\gamma_1 - r_1)e + \bar x \one_{[T-\frac{\bar x - \rvtxt{(c)}{(c)}{\phi_1(\bm\xi')}
(T)}{r_1-\gamma_1}
\rvtxt{(c)}{(c)}{,T}
]}$ and 
\rvtxt{(2)-5}{(2)}{%
$\bar x = \sum_{j=1}^\infty x^{(j)}$%
}%
.
Note that $\psi_1(\bm\xi'') 
\rvtxt{(c)}{(c)}{\ \geq\ } 
\psi_1(\bm\xi')$ and
$\psi_1(\bm\xi'')(T) = \psi_1(\bm\xi')(T)$%
\rvtxt{(c)-3}{(c)}{.  
To see this, let $T' \triangleq T-\frac{\bar x - \phi_1(\bm\xi')(T)}{r_1-\gamma_1}$.
Note that 
$$
T' %= T-\frac{\bar x - \phi_1(\bm\xi')(T)}{r_1-\gamma_1} 
=
T-\frac{\bar x - \xi_1(T) + \xi_1^\downarrow(T)}{r_1-\gamma_1}
=
T-\frac{\bar x - (\gamma_1-r_1)T - \bar x + \xi_1^\downarrow(T)}{r_1-\gamma_1} 
= -\frac{\xi_1^\downarrow(T)}{r_1-\gamma_1}.
$$
From the construction of $\xi_1'$, it ``attains'' its infimum at $T'-$, and hence,
$
(\xi_1')^\downarrow(t) = \xi_1'(T'-) = T' (\gamma_1-r_1) = \xi_1^\downarrow(T)
$ for $t \in [T', T]$. 
Note also that from the forms of $\xi_1$ and $\xi_1'$, we clearly have $\xi_1'(t) \leq \xi(t)$ for $t\in [0,T']$. 
Therefore, $(\xi_1')^\downarrow \leq \xi_1^\downarrow$ and $(\xi_1')^\downarrow(T) = \xi_1^\downarrow(T)$. 
Since $\psi_1(\bm\xi'') = -(\xi_1')^\downarrow$ and $\psi_1(\bm\xi') = -\xi_1^\downarrow$, we obtain the relationships between $\psi_1(\bm\xi'')$ and $\psi_1(\bm\xi')$ claimed above.
Now}%
, again from Lemma~\ref{comparison of one dimensional reflection}, we get $\phi_2(\bm\xi'')(T)\geq \phi_2(\bm\xi')(T)$.
Note that we constructed $\bm\xi''$ in such a way that $\phi_1(\bm\xi'')(T) = \phi_1(\bm\xi')(T)=\phi_1(\bm\xi)(T)$.
Note also that $\tilde I^{(d)}(\bm\xi'') \leq \tilde I^{(d)}(\bm\xi')$.
We arrive at the conclusion of the lemma by setting $\tilde {\bm \xi} = \bm\xi''$.

\CR{Note that the buffer content of the second queue is completely determined by the locations and lengths of the busy cycles of the first queue (regardless of the shapes).
Within a single busy cycle of the first queue, merging all the jumps (with the first one in the cycle) doesn't change the location and length of the cycle while reducing the value of the rate function, and hence, we can focus on the cases where the busy cycles in the first queue consist of single jumps. 
Note also that merging two busy cycles by shifting the end of the earlier cycle to the beginning of the subsequent one can only increase (or doesn't affect) the buffer content of the second queue at time $T$. 
Since $\xi$ in the effective domain of $\tilde I^{(d)}$ can only have countable jumps and the sum of the jump sizes sum up to a finite number, iterating this procedure produces a $\xi^*$ that has a single jump and optimizes $\tilde I^{(d)}$ due to the continuity of $\phi$.}

\end{proof}

\subsection{Proof of Proposition~\ref{j1-product-uniform-continuity-regulator} and Theorem~\ref{lipschitz-continuity-phi}
}
\label{subsec:Proof of Proposition j1-product-uniform-continuity-regulator}
Recall that $\prod_{i=1}^{d}\D[0,T]$ is the Skorokhod space equipped with the product $J_1$ topology and $\D^{\uparrow}[0,T] \triangleq \{\xi \in \D[0,T]: \xi \  \text{is non-decreasing on $[0,T]$ and $\xi(0) \geq 0$}\}.$ 
% By \cite{whitt2002stochastic} (p.\ 486)\CR{Is this the correct page?},\M{the subspace of nondecreasing functions is closed. By \cite{jacod2013limit}, the subspace of nondecreasing cadlag functions null at the origin is  a closed subspace of $\D$}
% \CR{That is not the same space as $\D^{\uparrow}[0,T]$.} 
$\D^{\uparrow}[0,T]$ is a closed subspace of $\D[0,T]$ w.r.t.\ the $J_1$ topology.
Hence, $\prod_{i=1}^{d}\D^{\uparrow}[0,T]$ is a closed subspace of $\prod_{i=1}^{d}\D[0,T]$ w.r.t.\ the product $J_1$ topology. Since $\D^{\beta}[0,T]$ is the  image of $\D^{\uparrow}[0,T]$ under the homeomorphism $\Upsilon^{\bm\beta}$, we have that $\prod_{i=1}^{d}\D^{\beta_i}[0,T]$
is a closed subset of $\prod_{i=1}^{d}\D[0,T]$.
\subsubsection{Some supporting lemmas}

\begin{lemma}\label{existence-of-almost-odentity-parameterizations}
Suppose that $\lambda, \mu\in \Lambda[0,T]$.
% are  strictly increasing homeomorphisms such that 
% \begin{itemize}
% \item[i)]$\lambda(0)=\mu(0)=0$, 
% \item[ii)] $\lambda(T)=\mu(T)=T$,
% \end{itemize}
Then, $\|\lambda\circ \mu-e\| \leq \|\lambda - e\| + \|\mu - e\|$.
\end{lemma}
\begin{proof}
$\|\lambda \circ \mu-e\| = \|\lambda  -\mu^{-1}\| \leq \|\lambda-e\| +\| \mu^{-1}-e\|=\|\lambda-e\| +\| e-\mu\|\leq 2\delta$.
\end{proof}
We now consider properties of continuous and increasing time deformations $w_i, \ i=1,\ldots,d$.

\begin{lemma}\label{min-max-time-deformations}
If $w_i\in \Lambda[0,T]$ for each $i=1,\ldots,d$, then $\hat{w}(s)=\min\{w_1(s),\ldots,w_d(s)\}$ and 
$\check{w}(s)=\max\{w_1(s),\ldots,w_d(s)\}$ 
%\CR{$\check{}$ and $\hat{}$ resemble $\vee$ and $\wedge$, but we use them in $\hat{w}$ and $\check{w}$ in the opposite way.}
% are increasing, continuous, $\check{w}(0)=\hat{w}(0)=0$, and $\check{w}(T)=\hat{w}(T)=T$.
also belong to $\Lambda[0,T]$.
\end{lemma}
\begin{proof}
The $\min$ and $\max$ of continuous and increasing functions are increasing and continuous. The other properties are easily verified.
\end{proof}

%Recall that $\|\boldsymbol{\beta}\|_{\infty}=\max_{1 \leq i \leq d }\beta_i$.
%BZ: has been defined 0.5 page before

Recall that $\psi$ is Lipschitz continuous w.r.t.\ $\|\cdot\|$ (Theorem 14.2.5 of \cite{whitt2002stochastic}). 
Let $K$ denote the Lipschitz constant of $\psi$ w.r.t.\ $\|\cdot\|$, which only depends on $Q$; in paticular, $K$ doesn't depend on $T$. 
\CR{we use (iii) of Theorem 14.2.6 of Whitt throughout this section without mentioning it. Need to mention it at least when used for the first time.}
\begin{lemma}\label{lemma-uniform-supremum-norm}
Let $\rvtxt{(3)-8}{(3)}{\bm \beta} = (\beta_1,\ldots,\beta_d) \in \R^d$ and $\bm\zeta \in \prod_{i=1}^{d}\D^{\beta_i}[0,T]$.
For any $w \in \Lambda[0,T]$, it holds that
 $$\|\psi(\bm\zeta)\circ w-\psi(\bm\zeta)\|< K \|\boldsymbol{\beta}\|_1\cdot \| w- e\|.$$
\end{lemma} 
\begin{proof}
Consider an arbitrary $s \in [0,T]$. 
% We will bound $|\psi(\zeta)(w(s)) \geq \psi(\zeta)(s)|.$
If $w(s) \geq s$, since $\psi(\bm\zeta)$ is an increasing function, $\psi(\bm\zeta)(w(s)) \geq \psi(\bm\zeta)(s).$ 
% Hence, we only need to bound $\psi(\zeta)(w(s))-\psi(\zeta)(s)$. 
Moreover, since $\bm\zeta \in \prod_{i=1}^{d}\D^{\beta_i}[0,T]$, 
$\bm\zeta$ has the following representation: $\bm\zeta(t) = \bm\xi(t)+\boldsymbol{\beta} \cdot t,$ where $\bm\xi \in \prod_{i=1}^d\D^{\uparrow}[0,T]$.
Consequently, for $t > u$, $\bm\zeta(t)-\bm\zeta(u) = \bm\xi(t)-\bm\xi(u)+\boldsymbol{\beta}\cdot(t-u) \geq \boldsymbol{\beta}\cdot(t-u).$ This implies that
\begin{equation}\label{construction-tilde-zeta-1}
\bm\zeta(w(s))=\bm\zeta((w(s)-s)+s) \geq \bm\zeta(s)+\boldsymbol{\beta}\cdot(w(s)-s). 
\end{equation}
Next, consider the path $\tilde{\bm\zeta}_1$ where 
\begin{equation*}
\tilde{\bm\zeta}_1(t)
=
\begin{cases}
\bm\zeta(t), & \ t \in [0,s], \\
\bm\zeta(s)+\boldsymbol{\beta}\cdot(t-s), & \ t \in [s,w(s)].
\end{cases}
\end{equation*}
Since $\tilde{\bm\zeta}_1 \leq \bm\zeta $ over $[0,w(s)]$, Result~\ref{generic-upperbound-for-psi} gives that $\psi(\tilde{\bm\zeta}_1)(w(s)) \geq \psi(\bm\zeta)(w(s))$. Furthermore,
let 
\begin{equation*}
\tilde{\bm\zeta}_2(t)
=
\begin{cases}
\bm\zeta(t),& \ t \in [0,s], \\
\bm\zeta(s),& \ t \in [s,w(s)].
\end{cases}
\end{equation*}
Then we have that $\psi(\tilde{\bm\zeta}_2)(w(s))=\psi(\bm\zeta)(s)$.  
Therefore,
\begin{align}
0\leq 
\psi(\bm\zeta)(w(s))-\psi(\bm\zeta)(s)   
&\leq 
\psi(\tilde{\bm\zeta}_1)(w(s))- \psi(\tilde{\bm\zeta}_2)(w(s))
\leq  
K \sup_{t \in [0,w(s)]}\|\tilde{\bm\zeta}_1(t)-\tilde{\bm\zeta}_2(t)\|_1
\nonumber\\
&\leq
K\|\boldsymbol{\beta}\|_1\cdot |w(s)-s|
\leq 
K\|\boldsymbol{\beta}\|_1\cdot \|w-e\|.
\label{lemma6.4-first-case}
\end{align}
Next, we consider the case $w(s) \leq s$. 
Since $\psi(\bm\zeta)$ is an increasing function, $\psi(\bm\zeta)(s) \geq \psi(\bm\zeta)(w(s))$. 
Furthermore, since $\bm\zeta \in \prod_{i=1}^{d}\D^{\beta_i}[0,T]$, we have that
\begin{equation}\label{construction-tilde-zeta-2}
\bm\zeta(s)=\bm\zeta((s-w(s))+w(s)) \geq \bm\zeta(w(s))+\boldsymbol{\beta}(s-w(s)). 
\end{equation}
Next, consider the path $\tilde{\bm\zeta}_1$, where
\begin{equation*}
\tilde{\bm\zeta}_1(t)
=
\begin{cases}
\bm\zeta(t), & \ t \in [0,w(s)], \\
\bm\zeta(s)+\boldsymbol{\beta}(t-w(s)), & \ t \in [w(s),s].
\end{cases}
\end{equation*}
Since $\tilde{\bm\zeta}_1 \leq \bm\zeta $ over $[0,s]$, Result~\ref{generic-upperbound-for-psi} gives that $\psi(\tilde{\bm\zeta}_1)(s) \geq \psi(\bm\zeta)(s)$. On the other hand,
let 
\begin{equation*}
\tilde{\bm\zeta}_2(t)
=
\begin{cases}
\bm\zeta(t),& \ t \in [0,w(s)], \\
\bm\zeta(s),& \ t \in [w(s),s].
\end{cases}
\end{equation*}
We then have that $\psi(\tilde{\bm\zeta}_2)(s)=\psi(\bm\zeta)(w(s))$.  
Therefore, 
\begin{align}
0
\leq 
\psi(\bm\zeta)(s)-\psi(\bm\zeta)(w(s))  
&
\leq 
\psi(\tilde{\bm\zeta}_1)(s)- \psi(\tilde{\bm\zeta}_2)(s)
\leq  K \sup_{t \in [0,s]}\|\tilde{\bm\zeta}_1(t)-\tilde{\bm\zeta}_2(t)\|_1
\nonumber\\
&
\leq
K\|\boldsymbol{\beta}\|_1\cdot|w(s)-s|
\leq
K\|\boldsymbol{\beta}\|_1\cdot\|w-e\|
.
\label{lemma6.4-second-case}
\end{align}
From \eqref{lemma6.4-first-case} and \eqref{lemma6.4-second-case}, we get (regardless of the value of $w(\cdot)$ at $s$)
$$\|\psi(\bm\zeta)(w(s)) -\psi(\bm\zeta)(s)\|_1 \leq  d K \|\boldsymbol{\beta}\|_1 \cdot\|w-e\| = K \|\boldsymbol{\beta}\|_1 \cdot\| w - e\| .$$
Taking the supremum over $s\in[0,T]$, we arrive at the conclusion of the lemma.
\end{proof}
Note that, if $\boldsymbol{\beta} =\bm 0$ and $\zeta \in \prod_{i=1}^{d}\D^{\beta_i}[0,T]$, then $\zeta$ belongs to $\prod_{i=1}^{d}\D^{\uparrow}[0,T]$ and  is non-negative at the origin. This implies $\psi (\zeta)=0$ and %$\psi(\zeta)(w)=0$ hence, 
the upper bound in Lemma \ref{lemma-uniform-supremum-norm} holds trivially.
Next, we state two more lemmas which are needed in our proof for the Lipschitz continuity of the regulator map in $\prod_{i=1}^{d}\D^{\beta_i}[0,T].$
Let $\iota \in \D[0,T]$ be \linkdest{nota-iota}$\iota(t) \equiv 1$, and \linkdest{nota-bm-iota}$\bm\iota = (\iota,\ldots,\iota) \in \prod_{i=1}^d \D[0,T]$.
 
\begin{lemma}\label{bound-multiple-time-deformations}
Consider $\bm{w}=(w_1,\ldots,w_d)$, each component of which is a time deformation in $\Lambda[0,T]$. Recall $\hat{w}$ and $\check w$ in Lemma~\ref{min-max-time-deformations}.
That is, $\check{w}(t)=\max\{w_1(t),\ldots,w_d(t)\}$, and $\hat{w}(t)=\min\{w_1(t),\ldots,w_d(t)\}$. 
Define the vector valued functions $\bm{\hat w}$, $\bm{\check w}$, and $\bm e$ from $[0,T]$ to $[0,T]^d$ as $\bm{\hat w} \triangleq (\hat w, \ldots, \hat w)$, $\bm{\check w} \triangleq (\check w, \ldots, \check w)$, and \linkdest{nota-bm-e}$\bm e \triangleq (e, \ldots, e)$.
For any $\xi \in \prod_{i=1}^{d}\D^{\beta_i}[0,T]$,
\begin{itemize}
\item[i)] $\psi(\xi_1\circ w_1,\ldots,\xi_d \circ w_d) \leq \psi(\bm\xi)\circ \hat{w}+ (d+1)K\|\boldsymbol{\beta}\|_{\infty}\cdot \|\bm w- \bm e\|\cdot \bm\iota$, and 
\item[ii)] $\psi(\xi_1\circ w_1,\ldots,\xi_d\circ w_d) + (d+1)K\|\boldsymbol{\beta}\|_{\infty}\cdot \|\bm w- \bm e\|\cdot \bm\iota \geq \psi(\bm\xi)\circ\check{w}$.
\end{itemize}
\end{lemma}
\begin{proof}
We start with \textit{i)}. 
Since $\bm\xi \in \prod_{i=1}^{d}\D^{\beta_i}[0,T]$ and $\hat{w}(s) \leq w_i(s)$, we have that for each $i=1,\ldots,d$,
\[
\xi_i(w_i(s)) \geq \xi_i(\hat{w}(s))-\|\boldsymbol{\beta}\|_{\infty}(w_i(s)-\hat{w}(s)), \ s \in [0,T]. 
\]
Note also that since $\big| \hat w (t) - e(t) \big| = \big| w_j (t) - e(t)\big|$ for some $j$,
$$
\|\bm e - \bm{\hat{w}}\| 
= 
\sup_{0\in[0,T]} \sum_{i=1}^d \big| \hat w (t) - e(t)\big|
\leq 
\sup_{0\in[0,T]} \sum_{i=1}^d \sum_{j=1}^d \big| w_j (t) - e(t)\big| = d \sup_{0\in[0,T]}\sum_{j=1}^d \big|w_j(t) - e(t)\big| = d \|\bm w - \bm e\|.
$$
Similarly, $\|\bm{\check{w}} - \bm e\| \leq d \|\bm w - \bm e\|.$
Therefore, due to Result~\ref{generic-upperbound-for-psi} and the Lipschitz continuity of $\psi$ w.r.t.\ $\|\cdot\|$,
\begin{align*}
\psi\big(\xi_1\circ w_1,\ldots,\xi_d\circ w_d\big)
&\leq 
\psi\big(\xi_1\circ\hat{w}-\|\boldsymbol{\beta}\|_{\infty}(w_1-\hat{w}),\ldots,\xi_d\circ\hat{w}-\|\boldsymbol{\beta}\|_{\infty}(w_d-\hat{w})\big)
\\
&
=
\psi\big(\bm\xi\circ\hat{w}-\|\boldsymbol{\beta}\|_{\infty}(\bm{w}-\bm{\hat{w}})\big)
\\
&
\leq 
\psi(\bm\xi\circ\hat{w}) +K\|\boldsymbol{\beta}\|_{\infty}\cdot \|\bm w - \bm{\hat{w}}\|\cdot \bm\iota
\\
&
\leq 
\psi(\bm\xi)\circ\hat{w}+K\|\boldsymbol{\beta}\|_{\infty}\cdot ( \|\bm w - \bm e\| +  \|\bm e - \bm{\hat{w}}\| )\cdot \bm\iota
\\
&
\leq 
\psi(\bm\xi)\circ\hat{w}+(d+1)K\|\boldsymbol{\beta}\|_{\infty}\cdot \|\bm w - \bm e\|\cdot \bm\iota.
\end{align*}
For \textit{ii)}, observe that $ \xi_i(\check{w}(s)) \geq \xi_i({w}_i(s))-\|\boldsymbol{\beta}\|_{\infty}(\check{w}(s)-w_i(s))$ 
for each $i=1,\ldots,d$ and $s \in [0,T]$,
since $\bm\xi \in \prod_{i=1}^{d}\D^{\beta_i}[0,T]$, and $\check{w}(s) \geq w_i(s)$ for each $i=1,\ldots,d$.
Therefore, due to Result~\ref{generic-upperbound-for-psi} and the Lipschitz continuity of $\psi$ w.r.t.\ $\|\cdot\|$,
\begin{align*}
\psi(\bm\xi)\circ \check{w}
&
=
\psi\big(\xi_1\circ \check{w},\ldots,\xi_d\circ \check{w}\big)
\\
&
\leq 
\psi\big(\xi_1\circ{w_1}-\|\boldsymbol{\beta}\|_{\infty}(\check{w}-w_1),\ldots,\xi_d\circ{w_d}-\|\boldsymbol{\beta}\|_{\infty}(\check{w}-w_d)\big)
\\
&
=
\psi\big((\xi_1\circ w_1,\ldots,\xi_d\circ w_d)-\|\boldsymbol{\beta}\|_{\infty}(\bm{\check{w}}-\bm w)\big)
\\
&
\leq 
\psi(\xi_1\circ w_1,\ldots,\xi_d\circ w_d)+K\|\boldsymbol{\beta}\|_{\infty}\cdot \|\bm{\check{w}}-\bm w\|\cdot \bm\iota
\\
&
\leq
\psi(\xi_1\circ w_1,\ldots,\xi_d\circ w_d)+K\|\boldsymbol{\beta}\|_{\infty}\cdot (\|\bm{\check{w}}-\bm e\| + \|\bm e - \bm{w}\|)\cdot \bm\iota
\\
&
=
\psi(\xi_1\circ w_1,\ldots,\xi_d\circ w_d)+(d+1) K\|\boldsymbol{\beta}\|_{\infty}\cdot \|\bm w - \bm e\|\cdot \bm\iota.
\end{align*} 
\end{proof}

 \begin{lemma}\label{the-desired-Lip-bound}
  For any $\bm\xi \in \prod_{i=1}^{d}\D^{\beta_i}[0,T]$ and $\bm{w}=(w_1,\ldots,w_d) \in \prod_{i=1}^d \Lambda[0,T]$, 
 $$
\|\psi(\xi_1\circ w_1,\ldots,\xi_d\circ w_d)-\psi(\bm\xi)\| \leq d(2d+1)K\|\boldsymbol{\beta}\|_1 \cdot \|\bm w-\bm e\|.
  $$
%  where $w^{(i)}, \ i=1,\ldots,d$, $\hat{w}$, and $\check{w}$ are defined in Lemma~\ref{min-max-time-deformations}.
 \end{lemma}

\begin{proof}
Due to Lemma~\ref{lemma-uniform-supremum-norm}, Lemma~\ref{bound-multiple-time-deformations}, and $\|\hat w - e\| \leq \|\bm w - \bm e\|$,
\begin{align*}
\psi(\xi_1\circ w_1,\ldots,\xi_d\circ w_d)-\psi(\bm\xi) 
&\leq 
\psi(\bm\xi)\circ\hat{w}-\psi(\bm\xi) + (d+1)K\|\boldsymbol{\beta}\|_{\infty}\cdot \|\bm w-\bm e\|\cdot \bm\iota
\\
&
\leq  
dK\|\boldsymbol{\beta}\|_1\cdot \|\hat w- e\|\cdot \bm\iota+(d+1)K\|\boldsymbol{\beta}\|_{\infty}\cdot\|\bm w-\bm e\|\cdot \bm\iota
\\
&
\leq
(2d+1)K\|\boldsymbol{\beta}\|_1 \cdot \|\bm w-\bm e\|\cdot \bm\iota.
\end{align*}
Similarly,
\begin{align*}
\psi(\bm\xi)-\psi(\xi_1\circ w_1,\ldots,\xi_d\circ w_d)
&\leq
\psi(\bm\xi)-\psi(\bm\xi)(\check{w})+(d+1)K\|\boldsymbol{\beta}\|_{\infty}\cdot \|\bm w-\bm e\|\cdot \bm\iota
\\
&\leq 
dK \|\bm \beta\|_1 \cdot \|\check w - e\|\cdot \bm\iota+(d+1)K\|\boldsymbol{\beta}\|_{\infty}\cdot \|\bm w-\bm e\|\cdot \bm\iota
\\
&\leq
(2d+1)K\|\boldsymbol{\beta}\|_1\cdot \|\bm w-\bm e\|\cdot \bm\iota.
\end{align*}
From these, the conclusion of the lemma follows.
\end{proof}

% \newpage 
\subsubsection{Lipschitz continuity of the reflection map} 
 
Now, we are ready to conclude Section~\ref{subsec:Proof of Proposition j1-product-uniform-continuity-regulator} with the proofs of Proposition~\ref{j1-product-uniform-continuity-regulator} and Theorem~\ref{lipschitz-continuity-phi}, which are the Lipschitz continuity of the regulator map and the buffer content component map, respectively, in the product $J_1$ topology. 
We start with the Lipschitz continuity of the regulator map $\psi$.
 
\linkdest{proof of j1-product-uniform-continuity-regulator} 
\begin{proof}[Proof of Proposition~\ref{j1-product-uniform-continuity-regulator}]
% We  prove that for any $\epsilon'>0$ there exists a $\delta'>0$ such that for any $\xi,\zeta$ with $d_p(\xi,\zeta)< \delta'$ implies $d_p(\psi(\xi),\psi(\zeta)) <\epsilon'$. Pick $\epsilon'>0$, and pick $\epsilon < \epsilon'/(d(K+1))$. Moreover, by Lemma~\ref{lemma-uniform-supremum-norm}, choose $\delta < \epsilon$ so that $\|\psi(\zeta)(w)-\psi(\zeta)\|_{\infty}<\epsilon$ if $\|w-e\|_{\infty}< \delta$. Let $\delta' < \delta/2$. 
Given $\bm\xi, \bm\zeta\in \prod_{i=1}^d \D^{\beta_i}[0,T]$, consider an arbitrary $\delta$ such that $d_p(\bm\xi,\bm\zeta)<\delta$. 
Then, there exists $\lambda_i\in \Lambda[0,T]$ such that $\|\xi_i\circ \lambda_i -\zeta_i \| \vee \|\lambda_i-e\| < \delta$ for each $i=1,\ldots,d$.  
Notice that
\begin{align}
d_p(\psi(\bm\xi),\psi(\bm\zeta)) 
& 
\leq  
\sum_{i=1}^{d}\inf_{ w_i \in \Lambda[0,T]}\big\|\psi_i(\bm\xi)\circ w_i -\psi_i(\bm\zeta)\big\| \vee \big\|w_i-e\big\|  \nonumber
\\
&\leq 
\sum_{i=1}^{d}\big\|\psi_i(\bm\xi)\circ \lambda_i -\psi_i(\bm\zeta)\big\| \vee \big\| \lambda_i - e \big\| \nonumber
\\
&
\leq
\sum_{i=1}^{d}\big\|\psi_i(\bm\xi)\circ \lambda_i -\psi_i(\xi_1\circ \lambda_1,\ldots,\xi_d\circ \lambda_d)\big\|\vee \big\| \lambda_i - e \big\| \nonumber
\\
&
\qquad\qquad
+
\sum_{i=1}^{d}\big\|\psi_i(\xi_1\circ \lambda_1,\ldots,\xi_d\circ \lambda_d)-\psi_i(\zeta_1,\ldots,\zeta_d)\big\|\vee \big\| \lambda_i - e \big\| \nonumber
\end{align}
Note that from Lemma~\ref{the-desired-Lip-bound},
\begin{align*}
&
\big\|\psi_i(\bm\xi)\circ \lambda_i -\psi_i(\xi_1\circ \lambda_1,\ldots,\xi_d\circ \lambda_d)\big\|
\\
&
=
\big\|\psi_i(\bm\xi) -\psi_i(\xi_1\circ \lambda_1,\ldots,\xi_d\circ \lambda_d)\circ  (\lambda_i)^{-1}\big\|
\\
&=
\big\|\psi_i(\bm\xi) -\psi_i(\xi_1\circ \lambda_1\circ  (\lambda_i)^{-1},\ldots,\xi_d\circ \lambda_d\circ  (\lambda_i)^{-1})\big\|
\\
&\leq
d(2d+1)K\|\bm\beta\|_1\cdot \big\|\bm e -\big(\lambda_1\circ  (\lambda_i)^{-1},\ldots,\lambda_d\circ  (\lambda_i)^{-1}\big)\big\|
\\
&\leq
d(2d+1)K\|\bm\beta\|_1\cdot \sum_{j=1}^d\big\|e -\lambda_j\circ  (\lambda_i)^{-1}\big\|
\\
&\leq
d(2d+1)K\|\bm\beta\|_1\cdot \sum_{j=1}^d\big(\big\|e -\lambda_j\big\|+\big\|e - (\lambda_i)^{-1}\big\|\big)
\\
&=
d(2d+1)K\|\bm\beta\|_1\cdot \sum_{j=1}^d\big(\big\|e -\lambda_j\big\|+\big\|e - \lambda_i\big\|\big)
\\
&\leq 
2d^2(2d+1)K\|\bm\beta\|_1\cdot \delta,
\end{align*}
where the third inequality is due to Lemma~\ref{existence-of-almost-odentity-parameterizations}.
On the other hand,
\begin{align*}
\big\|\psi_i(\xi_1\circ \lambda_1,\ldots,\xi_d\circ \lambda_d)-\psi_i(\zeta_1,\ldots,\zeta_d)\big\|
&\leq
\big\|\psi(\xi_1\circ \lambda_1,\ldots,\xi_d\circ \lambda_d)-\psi(\zeta_1,\ldots,\zeta_d)\big\|
\\
&\leq
K\big\|(\xi_1\circ \lambda_1,\ldots,\xi_d\circ \lambda_d)-(\zeta_1,\ldots,\zeta_d)\big\|
\\
&=
K\sum_{i=1}^d\big\|\xi_i\circ \lambda_i-\zeta_i\big\|
\leq Kd\delta.
\end{align*}
Therefore,
\begin{align}
\label{upper-bound-d-product}
d_p(\psi(\bm\xi),\psi(\bm\zeta)) 
&
\leq 
d(2d^2(2d+1)K\|\bm\beta\|_1 + Kd\vee 1)\cdot \delta
\end{align}
Letting $\delta \downarrow d_p(\bm\xi,\bm\zeta)$ we obtain Lipschitz continuity of $\psi$ w.r.t. $d_p$.
\end{proof}

\begin{proof}[Proof of Theorem~\ref{lipschitz-continuity-phi}]
The Lipschitz continuity of the regulator map has been proven in Proposition~\ref{j1-product-uniform-continuity-regulator}. We only need to verify the Lipschitz continuity of the buffer content component map $\phi$.
Let  $\delta$ be such that $d_{p}(\bm\xi,\bm\zeta)< \delta$. 
Then, %$d_{J_1}(\xi_i,\zeta_i)< \delta$ for each $i=1\ldots,d$.
there exists $\lambda_i\in\Lambda[0,T]$ such that $\|\xi_i\circ \lambda_i - \zeta_i\| \vee \|\lambda_i - e\| \leq \delta$ for each $i=1,\ldots,d$. 
Note that $\phi_i(\bm\xi)=\xi_i+\psi_i(\bm\xi)-\sum_{j \in \{1,\ldots,d\} \setminus \{i\}}q_{ji}\psi_j(\bm\xi)$. Hence, 
\begin{align*}
&
d_{J_1}(\phi_i(\bm\xi),\phi_i(\bm\zeta)) 
\\
&
= \textstyle
d_{J_1}\big(\xi_i + \psi_i(\bm\xi)-\sum_{j \in \{1,\ldots,d\} \setminus \{i\}}q_{ji}\psi_j(\bm\xi),\ 
\zeta_i+\psi_i(\bm\zeta)-\sum_{j \in \{1,\ldots,d\} \setminus \{i\}}q_{ji}\psi_j(\bm\xi)\big) 
\\
&
\textstyle
\leq 
\big\| \xi_i\circ \lambda_i + \psi_i(\bm\xi)\circ \lambda_i-\sum_{j \in \{1,\ldots,d\} \setminus \{i\}}q_{ji}\psi_j(\bm\xi)\circ \lambda_i-
\zeta_i-\psi_i(\bm\zeta)+\sum_{j \in \{1,\ldots,d\} \setminus \{i\}}q_{ji}\psi_j(\bm\xi)\big\| \vee \big\| \lambda_i - e\big\|
\\
&
\textstyle
\leq
\big\| \xi_i\circ \lambda_i - \zeta_i\big\| \vee \delta
+ \big\|\psi_i(\bm\xi)\circ \lambda_i - \psi_i(\bm\zeta)\big\| \vee \delta
+\sum_{j \in \{1,\ldots,d\} \setminus \{i\}}
\|\psi_j(\bm\xi)\circ \lambda_i-\psi_j(\bm\xi)\big\|  \vee \delta
\end{align*}
Note that $\big\| \xi_i\circ \lambda_i - \zeta_i\big\|\leq \delta$ and 
$\|\psi_j(\bm\xi)\circ \lambda_i-\psi_j(\bm\xi)\big\|$ can be bounded by 
$2d^2(2d+1)K\|\bm\beta\|_1\cdot\delta$ the say way as in the proof of Proposition~\ref{j1-product-uniform-continuity-regulator}.
Since $d_{p}(\phi(\bm\xi),\phi(\bm\zeta)) \leq \sum_{i=1}^{d} d_{J_1}(\phi_i(\bm\xi),\phi_i(\bm\zeta))$, we have that $\phi$ is Lipschitz continuous in $\prod_{i=1}^{d}\D^{\beta_i}[0,T]$ by letting $\delta \downarrow d_p(\bm\xi,\bm\zeta)$.
\end{proof}
\appendix

\section{Continuity  of some useful functions}
\label{appendix-continuity}
In this appendix, we include the proofs of continuity of some  functions in the product $J_1$ topology. Recall the function $\Upsilon^{\bm\beta}: \prod_{i=1}^{d}\D[0,T] \to \prod_{i=1}^{d}\D[0,T]$ where $\Upsilon^{\bm\beta}(\bm\xi)(t)=\bm\xi(t)+\bm\beta \cdot t$ for $t \in [0,T]$.
%and $\pi: \prod_{i=1}^{d}\D[0,T] \to \R^d$ where $\pi(\bm\xi)=\bm\xi(T)$. 

\linkdest{proof of Continuity-Upsilon-k-multi}
 \begin{proof}[Proof of Lemma~\ref{Continuity-Upsilon-k-multi}.]
For \textit{i)}, suppose that $\bm\xi$ and $\bm\zeta$ are given.
For each  $i \in \{1,\ldots,d\}$, let $\lambda_i$ be a homeomorphism such that $ \|\xi_i-\zeta_i \circ \lambda_i\| \vee \|\lambda_i-e\| < 2\cdot d_{J_1}(\xi_i, \zeta_i).$
Then, 
\begin{align}
d_{J_1}\left(\Upsilon^{\bm\beta}_i(\bm\xi),\Upsilon^{\bm\beta}_i(\bm\zeta)\right)
&
\leq 
\| \Upsilon^{\bm\beta}_i(\bm\xi) - \Upsilon^{\bm\beta}_i(\bm\zeta)\circ\lambda_i \|\vee \|\lambda_i - e \|
\\
&
= \|\xi_i-\zeta_i\circ \lambda_i-\beta_i(\lambda_i-e)\| \vee \|\lambda_i-e\| \nonumber
\\
&
\leq \|\xi_i-\zeta_i\circ \lambda_i\| \vee \|\lambda_i-e\| + \|\beta_i(\lambda_i-e)\| \vee \|\lambda_i-e\| \nonumber
\\
&
\leq 
2(1+1\vee |\beta_i|) \cdot d_{J_1}(\xi_i, \zeta_i).
\end{align} 
Consequently,  
\begin{align*}
d_p(\Upsilon^{\beta}(\bm\zeta),\Upsilon^{\beta}(\bm\xi)) 
&
= 
\sum_{i=1}^{d}d_{J_1}(\Upsilon_i^{\bm\beta}(\bm\zeta),\Upsilon_i^{\bm\beta}(\bm\xi))
\leq 
\sum_{i=1}^{d}2(1+1\vee |\beta_i|) \cdot d_{J_1}(\xi_i, \zeta_i)
\\
&\leq 
2(1+1\vee\|\bm\beta\|_1) \cdot d_p(\bm\xi,\bm\zeta).
\end{align*}

For \textit{ii)}, note that  $(\Upsilon^{\bm\beta})^{-1}(\bm\zeta)=\bm\zeta-\bm\beta \cdot e =\Upsilon^{-\bm\beta}(\bm\zeta)$, and hence, $\Upsilon^{\bm\beta}$ is injective and surjective. 
From this, the continuity of $(\Upsilon^{\bm\beta})^{-1}$ is also an immediate result of \textit{i)}.
\end{proof}

Finally, we prove that the  projection map is Lipschitz continuous in the product $J_1$ topology.

\linkdest{proof of continuity-of-the-projection-map}
\begin{proof}[Proof of Lemma~\ref{continuity-of-the-projection-map}.]
Let $\bm\xi, \bm\zeta \in \prod_{i=1}^{d}\D[0,T]$ be given.
Note first that 
\[
|\xi_i(T)-\zeta_i(T)|
= 
|\xi_i(T)-\zeta_i(\lambda(T))|
\leq \|\xi_i-\zeta_i\circ\lambda\|
\]
for any $\lambda \in \Lambda[0,T]$ since $\lambda(T) = T$.
Taking infimum over all $\lambda \in \Lambda[0,T]$, we see that $|\xi_i(T)-\zeta_i(T)| \leq d_{J_1}(\xi_i, \zeta_i)$.
Therefore,
\[
|\bm b^\intercal \bm\xi(T) - \bm b^\intercal \bm\zeta(T)| 
\leq \sum_{i=1}^d |b_i| \cdot |\xi_i(T) - \zeta_i(T)| 
\leq \sum_{i=1}^d |b_i| \cdot d_{J_1}(\xi_i, \zeta_i)
\leq \|\bm b\|_1 \cdot d_p(\bm\xi, \bm\zeta).
\]
%we are not writing a master thesis to make it appear look longer. if it fits on a single line we keep it on a single line. 
% Since this is true for any $\lambda\in \Lambda$, we have that
%  $|\pi_i(\xi)-\pi_i(\zeta)| \leq  d_{J_1}(\xi_i,\zeta_i)$ for each $i=1,\ldots,d$.
%  Therefore, 
%  we have that $\|\pi(\xi)-\pi(\zeta)\|_1 \leq d\cdot d_p(\xi,\zeta)$.
 \end{proof}

\section{Some useful tools on large deviations}

In this appendix, we include results that facilitate the use of the extended LDP. Given that the probability measures of $(X_n)$ satisfy the extended LDP in a metric space $\left(\mathcal{X},d\right)$, our results include the derivation of the extended LDP in closed subspaces of $\mathcal{X}$, and a variation of the contraction principle for Lipschitz continuous maps. Let $\mathcal{D}_I \triangleq \{x \in X: I(x)<\infty\}$ denote the effective domain of $I$.

\begin{lemma}\label{E-LDP-on-subspaces-full-measure}
Let $E$ be a closed subset of $\mathcal{X}$ and let $X_n$ be such that $\P(X_n \in E)=1$ for
all $n \geq 1$. Suppose that $E$ is equipped with the topology induced by $\mathcal{X}$. Then,
if the probability measures of $(X_n)$ satisfy the extended LDP in $(\mathcal{X},d)$ with speed $a_n$, and with rate function $I$ so that $\mathcal{D}_I \subseteq E$, then
the same extended LDP holds in $E$. 
\end{lemma}

\begin{proof}
% In the topology induced on $E$ by $\mathcal{X}$, the open sets are sets
% of the form $G \cap E$ with $G \subseteq \mathcal{X}$ open. Similarly, the closed sets in this
% topology are the sets of the form $F \cap E$ with $F \subseteq \mathcal{X}$ closed. Furthermore,
% $\P(X_n \in \Gamma) = \P(X_n \in \Gamma \cap E)$ for any $\Gamma \in \mathcal{B}$, where $\mathcal{B}$ is the Borel sigma-algebra.
% Suppose that an LDP holds in $E$, which is a closed subset of $X$. Extend
%the rate function $I$ to be a lower semicontinuous function on $X$ by setting
%$I(x) = \infty$ for any $x \in E^c$. Thus, $\inf_{x \in (\Gamma)^{\epsilon}}I(x) = \inf_{x \in (\Gamma)^{\epsilon} \cap E} I(x)$ for any $\Gamma \subseteq X$
%and the large deviations lower (upper) bound holds.
Suppose that an extended LDP holds in $\mathcal{X}$.
%  Since $E$ is closed, then $D_I \subset E$ by the
%large deviations lower bound (since $\P_n(X_n \in E^c)=0$ for all $n >0$ and $E^c$ is open).
For the upper bound, let $F$ be a closed subset of $E$ so that $F = F'\cap E$ for some $F'$ that is a closed subset of $\mathcal X$. 
Then, $F$ is a closed subset of $\mathcal{X}$. Hence,
$
%\begin{align*}\label{upper-bound-extended-ldp-on-subsspaces}
% &
% \limsup_{n \to \infty}\frac{1}{a_n}\log\P\left(X_n \in F \cap E\right)
% \\
% &
%  =
 \limsup_{n \to \infty}\frac{1}{a_n}\log\P\left(X_n \in F \right) \leq -\inf_{x \in F^{\epsilon}} I(x)= -\inf_{x \in F^{\epsilon} \cap E} I(x).
%\end{align*}
$
Next, for the lower bound, let $G$ be an open subset of $E$. That is, $G=G' \cap E$, where $G'$ is an open subset of $\mathcal{X}$. Then,
\rvtxt{6}{6}
{
\begin{align*}
&
\liminf_{n \to \infty}\frac{1}{a_n}\log\P\left(X_n \in G \right) 
%  =
%  \liminf_{n \to \infty}\frac{1}{a_n}\log\P\left(X_n \in G' \cap E\right)
%  \\
%  &
=
\liminf_{n \to \infty}\frac{1}{a_n}\log\P\left(X_n \in G'\right)
\geq 
-\inf_{
{x \in G'}
} I(x)
= -\inf_{x \in G} I(x).
\end{align*}
}
The level sets $\Psi_I(\alpha) \subseteq \mathcal{X}$ are closed, so $I$ restricted to $E$ remains lower semicontinuous.
 \end{proof}

We continue with a useful lemma on pre-images of Lipschitz continuous maps on metric spaces. 
% For a closed subset of the metric space $\left(\mathcal{X},d\right)$, recall that 
%  $A^{\epsilon} \triangleq \{\xi\in \mathcal{X}: d(\xi, A) \leq \epsilon \}$, where $d(\xi, A) = \inf_{\zeta \in A} d(\xi, \zeta)$.
 
\begin{lemma}\label{pre-images-lipschitz-maps}
Let $(\mathbb S, \sigma)$ and $(\mathbb{X}, d)$ be  metric spaces. Suppose that   $\mathrm{\Phi}: \left( \mathbb{X},d\right) \to \left(\mathbb S,\sigma \right)$ is a Lipschitz continuous mapping with Lipschitz constant $\|\mathrm{\Phi}\|_\lip$. Then, for any set $F \subset \mathbb {S}$, 
 it holds that 
\[
\left(\mathrm{\Phi}^{-1}(F)\right)^{\epsilon} \subseteq \mathrm{\Phi}^{-1}\left(F^{\,\epsilon\cdot\|\mathrm{\Phi}\|_\lip }\right).
\]
\end{lemma}
\begin{proof}
Let $\zeta \in \left(\mathrm{\Phi}^{-1}(F)\right)^{\epsilon}$. 
For each $n$, there exists $\xi_n$ such that $\xi_n \in \Phi^{-1}(F)$ and $d(\zeta,\xi_n) \leq \epsilon + 1/n$.
Note that 
$
\sigma(\Phi(\zeta), F)
\leq \sigma(\Phi(\zeta), \Phi(\xi_n)) 
\leq \|\Phi\|_\lip \cdot d(\zeta,\xi_n) 
\leq \|\Phi\|_\lip \cdot (\epsilon + 1/n)
$.
Taking $n\to\infty$, we have that $\sigma(\Phi(\zeta), F) \leq \epsilon\cdot\|\Phi\|_\lip $. 
That is, $\Phi(\zeta) \in F^{\|\Phi\|_\lip \epsilon}$, or $\zeta \in \Phi^{-1}\big(F^{\|\Phi\|_\lip \epsilon}\big)$. 
Since $\zeta$ was chosen arbitrarily from $\big(\Phi^{-1}(F)\big)^\epsilon$, we arrive at the desired inclusion.

\end{proof}

The following lemma is a version of the contraction principle adapted to the setting of extended LDP's. 

\begin{lemma}\label{alternative-yes-and-not-extended-contraction-principle}
Let $(\mathbbm{X}, d)$ and $(\mathbb S, \sigma)$ be  metric spaces. Suppose that the sequence of probability measures of $(\bm{X}_n)$ satisfies the lower and upper bounds of extended LDP in $(\mathbbm{X},d)$ with speed $a_n$ and a function $I$ (that is not necessarily a rate function). Moreover, let  $\mathrm{\Phi}: \left( \mathbbm{X},d\right) \to \left(\mathbb S,\sigma \right)$ be a Lipschitz continuous mapping and
set
$I'(y) \triangleq \inf_{\mathrm{\Phi}(x) = y}I(x).$ Then,
\begin{itemize}
\item[i)]
 $\Phi(\bm{X}_n)$ satisfies the following lower and upper bounds:
for any open set $G \subseteq \mathbb{S}$,
\[
\liminf_{n \to \infty}\frac{1}{a_n} \log \P\left(\Phi(\bm{X}_n)  \in G\right) \geq -\inf_{x \in G}I'(x),
\]
and for any closed set  $F \subseteq  \mathbb{S}$,
\[
\limsup_{n \to \infty}\frac{1}{a_n} \log \P\left(\Phi(\bm{X}_n)  \in F\right) \leq -\lim_{\epsilon \to 0}\inf_{x \in F^{\epsilon}}I'(x).
\]
\item[ii)] Suppose, in addition, that $I$ is a rate function and $\mathrm{\Phi}$ is a homeomorphism. Then, $I'$ is a rate function, and 
\rvtxt{7-1}{7}{$\Phi(\bm{X}_n)$}
satisfies the extended LDP in $(\mathbb{S},\sigma)$ with speed $a_n$ and rate function $I'$. 
\item[iii)]
If $I'$  is a good rate function---i.e., $\Psi_{I'}(M)\triangleq \{y\in \mathbb S: I'(y) \leq M \}$ is compact for each $M\in [0,\infty)$---then 
\rvtxt{7-2}{7}{$\Phi({\bm X}_n)$}  
satisfies the LDP in $(\mathbb S, \sigma)$ with speed $a_n$ and good rate function $I'$.
\end{itemize}

\end{lemma}

\begin{proof}
\textit{i)}
For  the upper bound, let $F$ be a closed subset of $\left(\mathbb{S},\sigma\right)$.
Thanks to Lemma~\ref{pre-images-lipschitz-maps}, for any $\epsilon>0$, we have that
$\left(\mathrm{\Phi}^{-1}(F)\right)^{\epsilon} \subseteq \mathrm{\Phi}^{-1}\left(F^{\,\epsilon\cdot\|\mathrm{\Phi}\|_\lip }\right)$.
Hence,
\begin{equation}\label{upperbound-of-the-upper-bound}
-\inf_{x \in \left(\mathrm{\Phi}^{-1}(F)\right)^{\epsilon} } I(x) 
\leq
-\inf_{x \in \mathrm{\Phi}^{-1}\left(F^{\, \epsilon\cdot\|\mathrm{\Phi}\|_\lip}\right)} I(x). 
\end{equation}
Furthermore, by the upper bound of the extended LDP of $\bm{X}_n$,  for any $\delta>0$ there exists an $n(\delta)$ such that for any $n \geq n(\delta),$
\begin{align}
%\label{inequality-extended-spldp-of-Xn}
% &
\P(\Phi(\bm{X}_n) \in F) 
% \nonumber\\
&=\P(\bm{X}_n \in \mathrm{\Phi}^{-1}(F)) 
\nonumber\\
&\leq 
\exp\left( a_n\left(-\inf_{x \in \left(\mathrm{\Phi}^{-1}(F)\right)^{\epsilon} } I(x)+ \delta\right)\right) 
%\ \text{for all } \ \epsilon>0.
% \end{equation}
% Consequently, (\ref{upperbound-of-the-upper-bound}), and (\ref{inequality-extended-spldp-of-Xn}) lead to 
% \begin{equation}
\nonumber\\
&
\label{inequality-extended-spldp-of-Sn}
% \P(\bm{X}_n \in \mathrm{\Phi}^{-1}(F)) 
\leq \exp\left(a_n\left(-\inf_{x \in \mathrm{\Phi}^{-1}\left(F^{\,\epsilon\cdot \|\mathrm{\Phi}\|_\lip }\right)} I(x)+ \delta\right)\right),
\end{align} 
 $\text{for any} \ n \geq n(\delta) \ \text{and} \ \epsilon >0$. 
% Next, for $n \geq n(\delta)$,
% \[
% \P\left( {\bm S}_n \in F\right)
% =
% \P\left(\mathrm{\Phi}\left( {\bm X}_n \right)\in F\right)
% =
%   \P\left( {\bm X}_n\in \mathrm{\Phi}^{-1}\big(F\big)\right)
% \leq
% \exp\left[a_n\left(-\inf_{x \in \mathrm{\Phi}^{-1}\left(F^{ \epsilon \|\mathrm{\Phi}\|_\lip }\right)} I(x)+ \delta\right)\right]. \]
 Therefore, 
\begin{equation*}%\label{ineq:phi-Xbar-almost-upper-bound1}
\limsup_{n \rightarrow \infty}\frac{1}{a_n}\log\P\left( \Phi({\bm X}_n)  \in F\right) 
\leq
-\inf_{x \in \mathrm{\Phi}^{-1}\left(F^{\,\epsilon\cdot\|\mathrm{\Phi}\|_\lip}\right)}I(x) + \delta 
=
-\inf_{y \in F^{\, \epsilon \cdot\|\mathrm{\Phi}\|_\lip}}I'(y) + \delta.
\end{equation*}		 
Letting $\delta \to 0$ and then $\epsilon \to 0$, we arrive at the desired large deviation upper bound.

For the lower bound, consider an open set $G$. 
Since $\mathrm{\Phi}^{-1}(G)$ is open, 
\[
\liminf_{n \rightarrow \infty}\frac{1}{a_n}\log
\P\left(\Phi(\bm{X}_n)  \in G\right)
=
\liminf_{n \rightarrow \infty}\frac{1}{a_n}\log
\P\left({\bm X}_n \in \mathrm{\Phi}^{-1}(G)\right)
\ge -\inf_{y \in \mathrm{\Phi}^{-1}\left(G\right)}I(y)=-\inf_{x \in G}I'(x).
\]
\textit{ii)} Since the upper and lower bounds for the extended large deviation principle have been proved in $i)$, we only have to prove that $I'$ is lower semi-continuous. 
% Since  $I$ is a rate function, its level sets $\Psi_{I}(M)\triangleq \{x \in  \mathbb{X}: I(x) \leq M \}$ are closed for every $M>0$.  
To see this,
% The level sets of $I'$ are $\Psi_{I'}(M)\triangleq \{y\in \mathbb S: I'(y) \leq M \}$, for every $M >0$ . 
note first that $I'(y) = I(\Phi^{-1}(y))$, and hence, for any $M>0$,
\[
\{y\in \mathbb S: I'(y) \leq M \}
=\{y\in \mathbb S:I(\Phi^{-1}(y)) \leq M \}
=\{\mathrm{\Phi}(x): I(x) \leq M \}=\mathrm{\Phi}(\Psi_{I}(M)).
\]
Since $\mathrm{\Phi}$ is a homeomorphism the r.h.s.\ is closed. Hence, $\Phi(\bm{X}_n)$ satisfies the extended LDP. 
\\
\textit{iii)} 
From the standard argument---see, for example, the proof of Theorem 4.2.1 of \cite{dembo2010large}---$I'$ is a good rate function. From Lemma 4.1.6 of \cite{dembo2010large}, we obtain
$
\lim_{\epsilon \to 0}\inf_{y \in F^{\epsilon\|\mathrm{\Phi}\|_\lip }}I'(y)=\inf_{y \in F}I'(y).
$
Consequently,  
\begin{align*}
 \limsup_{n \rightarrow \infty}\frac{\log\P\left( {\bm S}_n \in F\right) }{a_n}
%&
% \leq
% -\lim_{\epsilon \to 0}\inf_{x \in \mathrm{\Phi}^{-1}\left(F^{\epsilon\|\mathrm{\Phi}\|_\lip}\right)}I(x) \nonumber
% \\
% &
\leq 
-\lim_{\epsilon \to 0}\inf_{y \in F^{ \epsilon \|\mathrm{\Phi}\|_\lip }}I'(y)=-\inf_{y \in F}I'(y).
\end{align*}
\end{proof}

\bibliographystyle{apalike}
\bibliography{bibliography}

\iftheoremtree
\newpage
\section*{Theorem Tree}
% \lipsum[2]
\begin{thmdependence}[leftmargin=*]
\thmtreenode{\complete}
    {Theorem}{samplepathldpofstochasticnetworks}
    {0.8}{$\bm Z_n$ satisfies the upper and lower bounds of extended LDP}

    \begin{thmdependence}
    \thmtreenode{\complete}
        {Theorem}{lipschitz-continuity-phi} 
        {0.7}{
        $\bf R: \xi \mapsto (\phi(\xi), \psi(\xi))$ is Lipschitz w.r.t.\ $d_p$ on $\prod_{i=1}^{d}\D^{\beta_i}[0,T]$
        }
        
        \begin{thmdependence}
        \thmtreenodeproof{\complete}
            {Proposition}{j1-product-uniform-continuity-regulator}
            {0.7}{
            $\psi$ is Lipschitz w.r.t.\ $d_p$ on $\prod_{i=1}^{d}\D^{\beta_i}[0,T]$
            }

            \begin{thmdependence}
            \thmtreenode{\complete}
                {Lemma}{the-desired-Lip-bound}
                {0.8}{
                $\|\psi(\xi_1\circ w_1,\ldots,\xi_d\circ w_d)-\psi(\bm\xi)\| \leq d(2d+1)K\|\boldsymbol{\beta}\|_1 \cdot \|\bm w-\bm e\|$
                }

                \begin{thmdependence}
                \thmtreenode{\complete}
                    {Lemma}{lemma-uniform-supremum-norm}
                    {0.8}{
                    $\|\psi(\bm\zeta)\circ w-\psi(\bm\zeta)\|< K \|\boldsymbol{\beta}\|_1\cdot \| w- e\|.$
                    }

\iffalse
\begin{thmdependence}
\thmtreenode{\complete}
    {Result}{generic-upperbound-for-psi}
    {0.5}{$\psi$ is non-increasing}
    \hfill(in Lemma~\ref{lemma-uniform-supremum-norm}, \ref{bound-multiple-time-deformations})
\end{thmdependence}
\fi

                \thmtreenode{\complete}
                    {Lemma}{bound-multiple-time-deformations}
                    {0.8}{
                    i) $\psi(\xi_1\circ w_1,\ldots,\xi_d \circ w_d) \leq \psi(\bm\xi)\circ \hat{w}+ (d+1)K\|\boldsymbol{\beta}\|_{\infty}\cdot \|\bm w- \bm e\|\cdot \bm\iota$\\
                    ii) $\psi(\xi_1\circ w_1,\ldots,\xi_d\circ w_d) + (d+1)K\|\boldsymbol{\beta}\|_{\infty}\cdot \|\bm w- \bm e\|\cdot \bm\iota \geq \psi(\bm\xi)\circ\check{w}$
                    }
                \end{thmdependence}
            
            \thmtreenode{\complete}
                {Lemma}{existence-of-almost-odentity-parameterizations}
                {0.7}{$\|\lambda\circ \mu-e\| \leq \|\lambda - e\| + \|\mu - e\|$}
            
            \end{thmdependence}
        \end{thmdependence}

    \thmtreenode
    {\complete}
    {Lemma}
    {alternative-yes-and-not-extended-contraction-principle}
    {0.85}
    {Lipschitz Continuous Mapping Principle for extended LDP}
    
        \begin{thmdependence}
        \item[\complete]
            Lemma~\ref{pre-images-lipschitz-maps}
            \thmsum{0.7}{$\left(\mathrm{\Phi}^{-1}(F)\right)^{\epsilon} \subseteq \mathrm{\Phi}^{-1}\left(F^{\,\epsilon\cdot\|\mathrm{\Phi}\|_\lip }\right)$}
        \end{thmdependence}
        
    % \item[\complete]
    %     Result~\ref{theorem:multi-d+1-content}
    %     \thmsum{0.85}{$\bm X_n$ satisfies extended LDP}

    \item [\complete]
        Result~\ref{theorem:multi-d+1-content}
        \thmsum{0.8}{$\bar {\bm X}_n$ satisfies extended LDP with speed $L(n)n^\alpha$ and rate function $\tilde I^{(d)}$}
        
        \begin{thmdependence}
        
        \item[\complete]
            Result~\ref{sample-path-ldp-for-Jn} 
            \thmsum{0.8}{$\bar {\bm J}_n$ satisfies extended LDP with speed $L(n)n^\alpha$ and rate function $I^{(d)}$}
            
            \begin{thmdependence}
            
            \item[\complete]
                Theorem~2.3 and Remark~2.2 in \cite{bazhba2020sample}
            \item[\complete]
                Lemma~\ref{E-LDP-on-subspaces-full-measure}
                \thmsum{0.8}{If $X_n$ satisfies extended LDP in $\mathcal X$ and $P(X_n\in E) = 1$ where $\mathcal D_I \subseteq E$, then $X_n$ satisfies the same extended LDP in $E$.}
            \end{thmdependence}
            
        \item[\complete]
            Lemma~\ref{Continuity-Upsilon-k-multi}
            \thmsum{0.8}{$\bm\xi \mapsto \bm\xi - \bm\beta\cdot e$ is Lipschitz w.r.t.\ $d_p$, and it is a homeomorphism}
            
        \thmtreeref
        {Lemma}
        {alternative-yes-and-not-extended-contraction-principle}

        \end{thmdependence}

    \end{thmdependence}

\bigskip

\item[\complete]
    Theorem~\ref{LDP-Y-end-of-time-horizon}
    \thmsum{0.8}{$\lim_{n \to \infty}\frac{1}{L(n)n^{\alpha}}\log \P(\bm b^\intercal \bm{Z}_n(T) \geq y)=-V_{\geqslant}^*(y).$}
        
    \begin{thmdependence}

    \item[\complete] 
        Lemma~\ref{continuity-of-the-projection-map}
        \linktoproof{proof of continuity-of-the-projection-map}
        \thmsum{0.8}{$\bm\xi \mapsto \bm b^\intercal \bm \xi(T)$ is Lipschitz w.r.t.\ $d_p$}

    \thmtreeref
        {Lemma}
        {alternative-yes-and-not-extended-contraction-principle}

    \thmtreeref
        {Theorem}
        {samplepathldpofstochasticnetworks}

    % \item[\issue] 
    %     Lemma~\ref{reduction-to-one-step-functions}
    %     \thmsum{0.85}{For any $\xi \in \prod_{i=1}^d \D^{(\bm{\gamma}-\mathcal Q\bm r)_i}_{<\infty}[0,T]$, there is $\tilde \xi \in \prod_{i=1}^d \D^{(\bm{\gamma}-\mathcal Q\bm r)_i}_{\leqslant1}[0,T]$ with lower $\tilde I^{(d)}$ value and larger buffer content value at $T$. {\color{red} [not true]}}
    
    \thmtreenodeproof
        {\complete}
        {Lemma}
        {convergence-of-quasi-variational-V}
        {0.4}
        {$x\mapsto V_{\geqslant}^*(x)$ is $\alpha$-H\"older continuous}
        
        \begin{thmdependence}
        \item[\complete] 
            Lemma~\ref{construction-of-extra-jump}
            \thmsum{0.85}{Adding jumps at the end of the time horizon increases $\tilde I^{(d)}$ only so much.}
            \begin{thmdependence}
            \item[\complete]
                Result~\ref{size-discontinuities-of-phi}
                \thmsum{0.77}{$\phi(\bm\xi)$ and $\psi(\bm\xi)$ are discontinuous only at the discontinuities of $\bm\xi$; if $\bm\xi$ has no downward jumps, $\psi(\bm\xi)$ is continuous.}
            \end{thmdependence}
        \end{thmdependence}
    
    \item[\complete]
        Lemma~\ref{equality-of-VbVbt}
        \linktoproof{proof of equality-of-VbVbt}
        \thmsum{0.2}{$V_{\geqslant}^*(x) = V_{>}^*(x)$}
        % \rvin{[fine after changing the definition of $V_\geqslant$ and $V_>$]}

        \begin{thmdependence}

        \thmtreeref
            {Lemma}
            {convergence-of-quasi-variational-V}

        % \item[\elsewhere]
        %     Lemma~\ref{convergence-of-quasi-variational-V}
        %     \linktoelsewhere{tree convergence-of-quasi-variational-V}

        \end{thmdependence}
    
    \end{thmdependence}

\bigskip

\thmtreenodeproof{\complete}
    {Lemma}{reduction-to-one-step-functions}
    {0.8}{
    In two-node tandem network, $\bm\xi^* \in \D_{\leqslant 1}^{(\bm\gamma - \mathcal Q \bm r)_1}[0,T]\times \D_{\leqslant 1}^{(\bm\gamma - \mathcal Q \bm r)_2}[0,T]$
    }
    \begin{thmdependence}
    \thmtreenode{\complete}
        {Lemma}{comparison of one dimensional reflection}
        {0.8}{
        If $\eta \geq \omega$ and $\eta(T) = \omega(T)$,  
        then $(\eta -\eta^\downarrow)(T) \leq (\omega -\omega^\downarrow)(T)$.}
    \end{thmdependence}

\bigskip

\item[\complete]
    Result~\ref{continuity-of-the-reflection-map}
    \thmsum{0.5}{$\bm Y = \psi(\bm X)$ and $\bm Z = \phi(\bm X)$}

\item[\complete]
    Result~\ref{generic-upperbound-for-psi}
    \thmsum{0.5}{$\psi$ is non-increasing}
    \hfill(in Lemma~\ref{lemma-uniform-supremum-norm}, \ref{bound-multiple-time-deformations})

\end{thmdependence}
\fi

\end{document}